\documentclass[10pt,a4paper]{article}

\usepackage[latin1]{inputenc}
\usepackage{amsfonts}
\usepackage[dvipsnames]{xcolor}
\usepackage{amsmath}
\usepackage{amsthm}
\usepackage{amscd}
\usepackage{amssymb}
\usepackage{nicefrac}
\usepackage{mathrsfs}
\usepackage{bbm}
\usepackage{bm}
\usepackage{stackrel}
\usepackage{dsfont}
\usepackage{url}
\usepackage{bbm}
\usepackage{dsfont}
\usepackage{pdfpages}
\usepackage[margin=2.5cm]{geometry}
\usepackage{subfigure}

\numberwithin{equation}{section}

\usepackage
[pagebackref,colorlinks,citecolor=Plum,urlcolor=Periwinkle,linkcolor=DarkOrchid]
{hyperref}

\usepackage[displaymath, mathlines]{lineno}

\newtheorem{theorem}{Theorem}[section]
\newtheorem{proposition}[theorem]{Proposition}

\newtheorem{lemma}[theorem]{Lemma}
\newtheorem{remark}[theorem]{Remark}

\newcommand{\RR}{\mathbb{R}}

\newcommand{\Nor}{\mathcal{N}}

\newcommand{\bs}{\boldsymbol}

\def\ba#1\ea{\begin{linenomath}\begin{align*}#1\end{align*}\end{linenomath}}
\def\ban#1\ean{\begin{linenomath}\begin{align}#1\end{align}\end{linenomath}}

\begin{document}
% \linenumbers

\title{Central limit theorems for the monkey walk\\ with steep memory kernel}
\author{Erion-Stelios Boci\thanks{Department of Mathematical Sciences, University of Bath, Claverton Down, BA2 7AY Bath, UK.\newline Email: e.boci/c.mailler@bath.ac.uk} \and C\'ecile Mailler$^*$}
\date{\today}
\maketitle

\begin{abstract}
The monkey walk is a stochastic process defined as the trajectory of a walker that moves on $\mathbb R^d$ according to a Markovian generator, 
except at some random ``relocation'' times at which it jumps back to its position at a time sampled randomly in its past, according to some ``memory kernel''.
The relocations make the process non-Markovian and introduce a reinforcement effect (the walker is more likely to relocate in a Borel set in which it has spent a lot of time in the past).
In this paper, we focus on ``steep'' memory kernels: in these cases, the time sampled in the past at each relocation time is likely to be quite recent. 
One can see this as a way to model the case when the walker quickly ``forgets'' its past.
We prove limit theorems for the position of the walker at large times, which confirm and generalise the estimates available in the physics literature.
\end{abstract}

\section{Introduction}
The ``monkey walk'' is a non-Markovian stochastic process that 
was first defined in the physics literature as a model for animal (monkeys, in particular) 
foraging behaviour. 
The original model of Boyer and Solis-Salas~\cite{BSS14} is a random walk on $\mathbb Z^d$ 
that evolves like the simple symmetric random walk except at some random ``relocation'' 
times at which it jumps to a site it has already visited in the past, 
chosen with probability proportional to the number of past visits at that site. 
Equivalently, at relocation times, 
the walker chooses a time uniformly 
at random in its past and jumps to the site it visited at that random time.
We call the intervals between relocation times the ``run-lengths''.
In the original model of~\cite{BSS14}, the run-lengths are i.i.d.\ random variables, 
geometrically distributed with some parameter~$q>0$.
Boyer and Solis-Salas~\cite{BSS14} proved a central limit theorem 
for the position of the walker at large time; 
they showed in particular that the variance of this position 
is or order $(\log t)^d$ at large time~$t$.
The monkey walk thus diffuses much slower than the simple symmetric random walk, 
which diffuses at speed~$t^d$.
This is because the random relocations, which make sites that have been visited often in the past more likely to be visited again in the future, introduce a reinforcement (rich-gets-richer) effect.

Boyer, Evans and Majumdar~\cite{BEM17} later 
generalised this model by adding memory, i.e.\ making the walker more likely to relocate 
to sites it visited more recently.
Indeed, in their model, at relocation times, the walker chooses a time in its past according to some possibly non-uniform probability distribution, 
and then jumps to the site where it was at that random time.
More precisely, the idea is that, if a relocation happens at time $t>0$, then
the random time chosen by the walker in its past has density $\mu(x)\mathrm dx/(\int_0^t \mu)$, where $\mu : [0,\infty) \to [0,\infty)$ is a non-negative function called the memory kernel. 
(We ask that $\mu$ satisfies $\int_0^t \mu>0$ for all $t>0$.)
Boyer, Evans and Majumdar~\cite{BEM17} 
showed a central limit theorem for the position of the walker at large times, 
for a large class of possible memory ``kernels''.

Mailler and Uribe-Bravo~\cite{MUB19} later generalised the model even further by allowing the underlying motion (the walk between relocation times) to be any Markov process (instead of the simple symmetric random walk), possibly in continuous time, and the run-lengths to be a sequence of i.i.d. random variables of any distribution (instead of the geometric distribution).
They also allowed the memory to be non-uniform as in~\cite{BEM17} and proved that, for a large class of memory kernels, ``{\it if the underlying Markov process satisfies some central limit theorem, and if the run-lengths distribution have moments of high-enough order, then the associated monkey walk also satisfies a central limit theorem}''.

For technical reasons, the results of~\cite{MUB19}
only hold for relatively ``flat'' memory kernels, i.e.\ only if the walker 
does not forget its past too fast.
More precisely, they consider two classes of memory kernels: 
\[\mu_1(x) = \frac\alpha x (\log x)^{\alpha-1} \mathrm e^{\beta (\log x)^\alpha}, \quad (\alpha>0, \beta\geq 0),\]
and
\[\mu_2(x) = \gamma\delta x^{\delta-1}\mathrm e^{\gamma x^{\delta}}, \quad(\gamma>0, \delta\in [0,\nicefrac12]).\]
The aim of this paper is to prove limit theorems for steeper memory kernels, i.e.\ for $\mu = \mu_2$ with $\delta>\nicefrac12$. We are able to prove precise asymptotic results for the position of the walker at large times, which, in particular, confirm the predictions of~\cite{BEM17} (made in the case when the underlying Markov process is the standard Brownian motion).
In~\cite{MUB19}, the proofs rely on the analysis of the ``genealogical tree of the runs''. 
When the memory kernel is flat enough, this tree is close enough to a random recursive tree. 
When the memory kernel becomes steeper, the tree becomes less and less ``fat''; 
in particular, as mentioned in~\cite{MUB19}, 
the fact that the last common ancestor of two nodes taken uniformly, 
independently at random in the tree has constant-order height is no longer true, 
we believe, when $\delta\geq \nicefrac12$.
Because of this, we are not able to prove convergence of the occupation measure for steep memory kernels, while this could be done in~\cite{MUB19} for flatter memory kernels.
%, and the proofs of~\cite{MUB19} cannot be generalised to this case.

The literature on the ``monkey walk'' extends much beyond the works of Boyer and Solis-Salas~\cite{BSS14}, Boyer, Evans and Majumdar~\cite{BEM17}, and Mailler and Uribe Bravo~\cite{MUB19}, which we have discussed so far. 
Indeed, Boyer and Pineda~\cite{BP16} proved limiting theorems for the position of the walker at large times in the case when the underlying Markov process is a random walk with heavy-tailed increments (a case which falls under the more general, more recent framework of~\cite{MUB19}).
Large deviations for the position of the walker at large times were established in the work of Boci and Mailler~\cite{BM}.
In~\cite{BFGM19}, Boyer, Falc{\'o}n-Cort{\'e}s, Giuggioli and Majumdar exhibit an interesting localisation phenomenon when the probability to relocate is larger at the origin than at the other sites of $\mathbb Z^d$.
Recently, Boyer and Majumdar~\cite{BM24} have introduced a continuous-time variant of the model in which, between relocations, the particle moves at constant speed on a one-dimensional line with a telegraphic noise (i.e.\ the sign of the particle's speed changes at constant rate). They get explicit formulas for the distribution of the position of the particle at all times and show large deviation results.

The case when the walker resets to a fixed position (eg.\ the origin) at relocation times has also been studied in the literature (see, e.g.\ \cite{EMS} for a literature review on the subject), but, unsurprisingly, it leads to a drastically different behaviour.

\subsection{Mathematical definition of the model and notation} 
The monkey walk $X = (X_t)_{t\geq 0}$ is a stochastic process on $\mathbb R^d$ whose distribution depends on three parameters: 
\begin{itemize}
\item a semi-group $P = (P_t)_{t\geq 0}$ on $\mathbb R^d$ (the distribution of the underlying Markov process);
\item a probability distribution $\phi$ on $[0,\infty)$ (the run-length distribution);
\item a function $\mu : [0,\infty) \to [0,\infty)$ such that $\int_0^t \mu >0$ for all $t>0$.
\end{itemize}
The process $X$ is defined as follows (see Figure~\ref{fig:def}): 
first sample $(L_n)_{n\geq 1}$ a sequence of i.i.d. random variables of distribution $\phi$ and let $T_n = \sum_{i=1}^n L_i$ for all $n\geq 0$. 
Then, let $(X(s))_{0\leq s<T_1}$ be the Markov process of semi-group $P$ started at the origin.
Then, for all $n\geq 1$, given $(X(s))_{0\leq s<T_n}$,
\begin{itemize}
\item let $R_n$ be a random variable on $[0, T_n)$ whose distribution has density $\mu/(\int_0^{T_n}\mu)$ on $[0,T_n)$,
\item let $(X(s))_{T_n\leq s<T_{n+1}}$ be the Markov process of semi-group $P$ started at $X(R_n)$.
\end{itemize}
For all $n\geq 1$, we call $L_n$ the ``length of the $n$-th run'' and $T_n$ the ``$n$-th relocation time''. 

\begin{figure}
\begin{center}
\includegraphics[width = 12cm, page= 12]{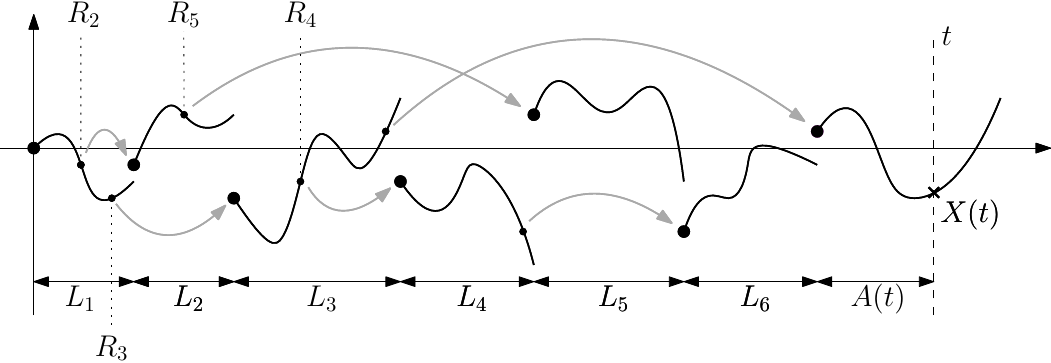}
\end{center}
\caption{A representation of the monkey walk. The grey arrows point from $R_n$ to $T_n$ for all $n\geq 1$; it is convenient to have this ``branching'' representation in minds for the proofs.}
\label{fig:def}
\end{figure}

In this paper, we assume that
\begin{equation}\label{eq:def_kernel}
\mu(x) = \gamma \delta x^{\delta-1}\mathrm e^{\gamma x^\delta}
\end{equation}
for some $\gamma>0$ and $\delta>\nicefrac12$.

\begin{remark} 
Note that this framework also includes the case of a discrete-time underlying process. Indeed, in that case, we make the underlying process defined on $[0,\infty)$ by making it constant on each interval $[n, n+1)$, $n\geq 0$, and we round the relocation times to the integer above.
\end{remark}

\subsection{Statement of the results}
Our main result holds under the following assumptions on the underlying Markov process of semi-group~$P$:
\begin{itemize}
\item[\textbf{(A1)}] There exist two functions $a : [0,\infty) \to \mathbb R^d$ and $b : [0,\infty) \to (0,\infty)$, and a probability distribution~$\nu$ on $\mathbb R^d$ such that, for all $z\in\mathbb R^d$, if $Z = (Z(t))_{t\geq 0}$ is the Markov process of semi-group $P$ started at~$z$, then
\[\frac{Z(t)-a(t)}{b(t)} \Rightarrow \nu,\]
in distribution as $t\uparrow\infty$.
(We say that $Z$ is $(a, b)$-ergodic, following the terminology used in~\cite{MUB19}, and before that in~\cite{MM17}.)
\item[\textbf{(A2)}] For all $x\in \RR$ and for all functions $\varepsilon : [0,\infty)\to \mathbb R$ satisfying $\varepsilon(t) = o(\sqrt t)$ as $t\uparrow\infty$, the following limits exist and are finite: \begin{equation} \label{eq: definition of f and g}
f(x)=\lim_{t\to \infty} \frac{a(t+x\sqrt{t}+\varepsilon(t))-a(t)}{b(t)}
\quad \text{and} \quad 
g(x)=\lim_{t\to \infty} \frac{b(t+x\sqrt{t}+\varepsilon(t))}{b(t)}.
\end{equation}
\item[\textbf{(A2$_{\bs\delta}$)}] 
For all functions $\varepsilon : [0,\infty)\to \mathbb R$ satisfying, as $t\uparrow\infty$, 
\[\varepsilon(t) = 
\begin{cases}
o(t^{\nicefrac32-\delta}) & \text{ if }\delta\in (1, \nicefrac32)\\
\mathcal O(1) & \text{ if }\delta>\nicefrac32
\end{cases},\] we have
\begin{equation} \label{eq: definition of alpha and beta}
0=\lim_{t\to \infty} \frac{a(t+\varepsilon(t))-a(t)}{b(t)}
\quad \text{and} \quad 
1=\lim_{t\to \infty} \frac{b(t+\varepsilon(t))}{b(t)}.
\end{equation}
\end{itemize}

\begin{remark}
The results of~\cite{MUB19} hold under Assumptions {\rm\bf (A1-2)}; 
recall that they cover the case of $\delta \in(0,\nicefrac12]$.
For $\delta\in (\nicefrac12, 1]$, we use the same assumptions.
However, for $\delta>1$, i.e.\ for the steepest memory kernels, 
we need to replace Assumption {\rm\bf (A2)} by {\rm\bf (A2$_{\bs\delta}$)}.
\end{remark}

To state our main result, we introduce the following notation:
For all $t\geq 0$, we set
\begin{equation}\label{eq:def_s}
s(t)
= \begin{cases}
\displaystyle \frac{\gamma\delta t^{\delta}}{\mathbb E [L]} 
\sum_{k=0}^{\lfloor \frac{\delta}{2 - 2\delta}\rfloor} 
\frac{(-\gamma\delta)^{k} \mathbb E[L^{k+2}]}{(k+2)!\big(\delta-(1-\delta)k\big)}
\cdot t^{k(\delta -1)}
& \text{ if }\delta\in(0, 1),\\[5pt]
\displaystyle\bigg(\mathbb E[L] - \frac{1-\mathbb E[\mathrm{e}^{-\gamma L}]}{\gamma}\bigg) t & \text{ if } \delta = 1,\\[5pt]
\displaystyle t-\frac{t^{2-\delta}}{\gamma \delta (2-\delta)} & \text{ if }\delta \in (1, 2],\\[5pt]
t & \text{ if }\delta>2.
\end{cases}
\end{equation}
Roughly speaking, our main result says that, in distribution, $X(t)\approx Z(s(t))$, 
where $Z$ is the Markov process of semi-group $P$ started at the origin.
One can see that, as $\delta$ increases, and thus as the memory kernel becomes steeper and steeper, $s(t)$ becomes larger and larger.

\begin{remark}\label{rk:123}
We comment on the formula for $s(t)$ when $\delta\in(\nicefrac12, 1)$.
For $\delta\in (\nicefrac12, \nicefrac23)$, the sum has only one summand. Indeed, $0\leq \delta/(2\delta -2)< 1$ if and only if $\delta\in (\nicefrac12, \nicefrac23)$. In that case,
\[s(t) = \frac{\gamma \mathbb E[L^2] }{2\mathbb E [L]}\cdot t^{\delta}.\]
Similarly, if $\delta\in [\nicefrac23, \nicefrac34)$, the sum has two summands, and
\[s(t) =  \frac{\gamma \mathbb E[L^2] }{2\mathbb E [L]} \cdot t^{\delta} -
\frac{(\gamma\delta)^2 \mathbb E[L^{3}]}{6\big(2\delta -1\big)\mathbb E [L]}
\cdot t^{2\delta -1}.
\]
More generally, for all integers $i\geq 1$, for all $\delta\in [\frac{i}{i+1}, \frac{i+1}{i+2})$,
the sum has $i$ summands, each of them of the form ``constant times a power of $t$'', where the power of $t$ deacreases with the index $k$ of the sum.
Note in particular that, for all $\delta\in (\nicefrac12, 1)$, as $t\uparrow\infty$,
\[s(t)\sim \frac{\gamma \mathbb E[L^2] }{2\mathbb E [L]}\cdot t^{\delta}.\]
\end{remark}

The following three theorems are our main results:
\begin{theorem}[Small values of $\delta$]\label{th:CLT_small} 
Let $X=(X(t))_{t\geq 0}$ be the monkey process of semigroup~$P$, 
run-length distribution~$\phi$ and memory kernel 
$\mu(x)=\gamma \delta x^{\delta -1}\mathrm{e}^{\gamma x^{\delta}}$, 
where $\gamma>0$ and $\delta\in (0,1)$.
Let $L$ be a random variable of distribution~$\phi$.
We assume that {\rm\bf (A1-2)} hold and assume that $\mathbb E[L^p]<\infty$, where
\begin{equation}\label{eq:def_p}
p=\max\Big\{8, \Big\lfloor \frac{1}{1-\delta}\Big\rfloor +1\Big\}.
\end{equation}
Then, 
%conditionally on the run-lengths $(L_i)_{i\geq 0}$, 
%and for almost all realisations of these run-lengths, 
in distribution as $t\uparrow\infty$,
\[\frac{X(t)-a(s(t))}{b(s(t))} \Rightarrow f(\Omega) + \Lambda g(\Omega),\]
where $(s(t))_{t\geq 0}$ is defined as in~\eqref{eq:def_s},
and where $\Omega \sim \mathcal N(0, {2\mathbb E[L^3]}/(3\mathbb E[L^2]))$ and $\Lambda\sim \nu$ are two independent random variables.
\end{theorem}

\begin{remark} 
We have included the case $\delta\in (0, \nicefrac12]$ in Theorem~\ref{th:CLT_small}.
This case is proved in~\cite{MUB19}. In this paper, we only prove Theorem~\ref{th:CLT_small} in the case $\delta\in(\nicefrac12, 1)$.
Although the general idea of the proof is the same as in~\cite{MUB19}, the proof is more technical because the expansion of $s(t)$ up to order $t^{\nicefrac\delta2}$ (which is the accuracy we need) has more and more terms as $\delta\uparrow1$.
%However, we include it here to show the whole range of behaviour in one statement. 
\end{remark}

\begin{theorem}[$\delta = 1$]\label{th:CLT_crit} 
Let $X=(X(t))_{t\geq 0}$ be the monkey process of semigroup~$P$, 
run-length distribution~$\phi$ and memory kernel 
$\mu(x)=\gamma \mathrm{e}^{\gamma x}$, 
where $\gamma>0$.
Let $L$ be a random variable of distribution~$\phi$.
We assume that {\rm\bf (A1-2)} hold and assume that $\mathbb E[L^4]<\infty$.
Then, in distribution when $t\uparrow\infty$,
\[\frac{X(t)-a(s(t))}{b(s(t))} \Rightarrow f(\Omega) + \Lambda g(\Omega),\]
where $s(t)$ is defined as in~\eqref{eq:def_s},
$\Lambda\sim\nu$ and $\Omega$ are independent, 
and $\Omega = \Omega_1 + \Omega_2 + \Omega_3$ is the sum of three dependent Gaussian random variables whose variances and covariances are as follow:
\[\mathrm{Var}(\Omega_1) 
= \frac1{\gamma^2}+\mathbb E\bigg[L^2-\bigg(L+\frac{\mathrm e^{-\gamma L}}{\gamma}\bigg)^{\!\!2}\bigg],
\quad\mathrm{Var}(\Omega_2) 
= \mathrm{Var}\bigg(L- \frac{1-\mathrm e^{-\gamma L}}\gamma\bigg)
= \mathrm{Var}\bigg(L + \frac{\mathrm e^{-\gamma L}}\gamma\bigg),\]
\[\mathrm{Var}(\Omega_3) 
= \frac{\mathrm{Var}(L)\mathbb E[L-(1-\mathrm e^{-\gamma L})/\gamma]^2}{\mathbb E[L^2]},\]
\[\mathrm{Cov}(\Omega_1, \Omega_2) 
= \mathrm{Var}\bigg(L-\frac{1-\mathrm e^{-\gamma L}}{\gamma}\bigg),
\mathrm{Cov}(\Omega_1, \Omega_3) 
= \mathrm{Cov}(\Omega_2, \Omega_3)
= \frac{\mathrm{Cov}(L, L+\mathrm e^{-\gamma L})\mathbb E[L-(1-\mathrm e^{-\gamma L})/\gamma]}{\gamma\mathbb E[L]}.
\]
\end{theorem}

\begin{theorem}[Large values $\delta$]\label{th:CLT_large} 
Let $X=(X(t))_{t\geq 0}$ be the monkey process of semigroup~$P$, 
run-length distribution~$\phi$ and memory kernel 
$\mu(x)=\gamma\delta x^{\delta-1} \mathrm{e}^{\gamma x^\delta}$, 
where $\gamma>0$ and $\delta>1$.
Let $L$ be a random variable of distribution~$\phi$.
We assume that {\rm\bf (A1)} and {\rm\bf (A2$_{\bs\delta}$)} hold and assume that $\mathbb E[L^4]<\infty$.
Then, 
%conditionally on the run-lengths $(L_i)_{i\geq 0}$, 
%and for almost all realisations of these run-lengths, 
in distribution as $t\uparrow\infty$,
\[\frac{X(t)-a(s(t))}{b(s(t))} \Rightarrow \nu.\]
\end{theorem}

\begin{remark}[On the moment conditions]
We summarise here our moment conditions on the run-length distribution, for increasing values of $\delta$:
\begin{itemize}
\item for $\delta<\nicefrac78$, we ask that $\mathbb E[L^8]<\infty$;
\item for $\delta\in [\nicefrac78, \nicefrac89)$ we ask that $\mathbb E[L^9]<\infty$; etc
\item for $\delta \geq 1$, we ask that $\mathbb E[L^4]<\infty$.
\end{itemize}
Interestingly, one can note that we ask for more and more control on the upper tail of the run-length distribution as $\delta\uparrow1$. We believe that this is a caveat of the expansion in the expression for $s(t)$; if the run-lengths had heavier tails, a limiting result would hold, but there would be no ``nice'' formula for $s(t)$.
%Also, for $\delta\in(1, \nicefrac32]$, and only in this region, we ask for some control on the tail near zero.
\end{remark}

\subsection{One example}
Before discussing our main result further, we apply it to a simple example:
let $\phi$ be the standard exponential distribution and $P$ be the semi-group of the one dimensional Brownian motion with constant drift equal to~1.
We first check Assumptions {\bf (A1-2)} and {\bf (A2$_{\bs\delta}$)}:
If $Z$ is the one dimensional Brownian motion with constant drift equal to~1, 
then, in distribution as $t\uparrow\infty$,
\[\frac{Z(t)-t}{\sqrt t} \Rightarrow \mathcal N(0,1).\]
In other words, Assumption {\bf (A1)} holds with $a : t \mapsto t$, $b : t\mapsto \sqrt t$, and $\nu = \mathcal N(0,1)$.
Also, for all $x\in\mathbb R$ and $t\geq 0$,
\[\frac{a(t+ x\sqrt t + \varepsilon(t))- a(t)}{b(t)}
= \frac{x\sqrt t + \varepsilon(t)}{\sqrt t} \to x,\]
as $t\uparrow\infty$, as long as $\varepsilon(t) = o(\sqrt t)$.
Similarly,
\[\frac{b(t+ x\sqrt t + \varepsilon(t))}{b(t)} = \sqrt{\frac{t+x\sqrt t + \varepsilon(t)}{t}} \to 1.\]
Thus, Assumption {\bf (A2)} holds with $f(x) = x$ and $g(x) = 1$, for all $x\in\mathbb R$.
Finally, if $\varepsilon(t) = o(\sqrt t)$, then
\[\frac{a(t+\varepsilon(t))-a(t)}{b(t)} = \frac{\varepsilon(t)}{\sqrt t} \to 0,\]
and
\[\frac{b(t+\varepsilon(t))}{b(t)} = \sqrt{\frac{t+\varepsilon(t)}{t}} \to 1,
\]
as $t\uparrow\infty$, which implies that Assumption {\bf (A2$_{\bs\delta}$)} also holds in this case.

Also note that, if $L$ is a standard exponential random variable, 
then $\mathbb E L^k = k! <\infty$ for all $k\geq 1$.
Thus the assumptions of Theorems~\ref{th:CLT_small}, \ref{th:CLT_crit} and~\ref{th:CLT_large} hold for all $\delta\in (0, \infty)$.
We now look at the conclusions of these theorems for different values of $\delta$.
For all of these, one can check that the diffusive speed of the monkey walk coincides with the variance estimates of~\cite[Section~V]{BEM17} (made for the same model: runs of Brownian motion and exponential run-lengths -- the only difference is that they consider the case with no drift, but this does not affect the variance estimates).

\noindent{\bf $\bullet$ Assume that $\delta \in (0, \nicefrac23)$.}
In that case (see Remark~\ref{rk:123}),
\[s(t) = \frac{\gamma \mathbb E[L^2] }{2\mathbb E [L]}\cdot t^{\delta}
 =\gamma t^{\delta}.\]
Thus, by Theorem~\ref{th:CLT_small},
\[\frac{X(t) - \gamma t^{\delta}}{\sqrt{\gamma} t^{\nicefrac\delta2}}
\Rightarrow \Omega + \Lambda,\]
where $\Omega\sim \mathcal N(0, 2)$ (because $\frac{2\mathbb E[L^3]}{3\mathbb E[L^2]} = 2$) and $\Lambda\sim\mathcal N(0, 1)$.
Because $\Omega$ and $\Lambda$ are independent, we get
\[\frac{X(t) - \gamma t^{\delta}}{\sqrt{\gamma} t^{\nicefrac\delta2}}
\Rightarrow \mathcal N(0, 3),\quad
\text{ or equivalently, }\quad
\frac{X(t) - \gamma t^{\delta}}{t^{\nicefrac\delta2}}
\Rightarrow \mathcal N(0, 3\gamma).\]

\noindent{\bf $\bullet$ Assume that $\delta \in (\nicefrac23, \nicefrac34)$.}
The only difference with the previous case is in the definition of $s(t)$:
as discussed in Remark~\ref{rk:123}, in this case,
\[s(t) = \frac{\gamma \mathbb E[L^2] }{2\mathbb E [L]} \cdot t^{\delta} -
\frac{(\gamma\delta)^2 \mathbb E[L^{3}]}{6\big(2\delta -1\big)\mathbb E [L]}
\cdot t^{2\delta -1}
= \gamma t^{\delta} + \frac{(\gamma\delta)^2}{2\delta -1 }\cdot t^{2\delta -1}.\]
Theorem~\ref{th:CLT_small} thus gives
\[\frac{X(t) - \gamma t^{\delta} - \frac{(\gamma\delta)^2}{2\delta -1}\cdot t^{2\delta -1}}
{t^{\nicefrac\delta2}} \Rightarrow \mathcal N(0,3\gamma).\]
It is important to note that one cannot neglect the second order term of $s(t)$ in the numerator above because $t^{2\delta -1}\gg t^{\nicefrac\delta2}$. Indeed, $2\delta -1>\nicefrac\delta2$ as soon as $\delta>\nicefrac23$.
This is the reason why, in the definition of $s(t)$ for $\delta<1$, 
we need to keep more and more terms in the sum as $\delta$ increases.

\medskip
\noindent{\bf $\bullet$ Assume that $\delta = 1$.}
In that case, Theorem~\ref{th:CLT_crit} applies. We have
\[s(t) = \bigg(\mathbb E[L]-\frac{1-\mathbb E[\mathrm e^{-\gamma L}]}{\gamma}\bigg)t
= \frac{\gamma}{\gamma+1}\cdot t.\]
Thus,
\[\frac{X(t) - {\gamma t}/(\gamma+1)}{\sqrt{{\gamma t}/(\gamma+1)}}\Rightarrow \Omega + \Gamma,
\]
where $\Gamma\sim \mathcal N(0,1)$ and $\Omega = \Omega_1+\Omega_2+\Omega_3$ are independent, with $\Omega_1, \Omega_2$, and $\Omega_3$ three Gaussian whose variances and covariances are given by
\[\mathrm{Var}(\Omega_1) = \frac{(\gamma-1)^2}{2\gamma^2(\gamma+1)^2},
\quad
\mathrm{Var}(\Omega_2) = \frac{\gamma^2(2\gamma+5)}{(\gamma+1)^2(2\gamma+1)},
\quad\mathrm{Var}(\Omega_3)  = \frac{\gamma}{\gamma+1},\]
\[\mathrm{Cov}(\Omega_1, \Omega_2) = \frac{\gamma^2(2\gamma+5)}{(\gamma+1)^2(2\gamma+1)},\quad\text{ and }\quad
\mathrm{Cov}(\Omega_1, \Omega_3)=\mathrm{Cov}(\Omega_2, \Omega_3)
= \frac{\gamma^2+2\gamma+2}{(\gamma+1)^3}.\]

\medskip
\noindent{\bf $\bullet$ Assume that $\delta>1$.}
In that case, Theorem~\ref{th:CLT_large} applies and gives that, if $\delta\in (1, \nicefrac32]$, then
\[\frac{X(t)-t+t^{2-\delta}/(\gamma\delta(2-\delta))}{\sqrt t} \Rightarrow \mathcal N(0,1),\]
and if $\delta>\nicefrac32$,
\[\frac{X(t)-t}{\sqrt t} \Rightarrow \mathcal N(0,1).\]
In the case when $\delta>\nicefrac32$, the memory kernel is so steep that the monkey walk satisfies the same central limit theorem as the standard Brownian motion (with no relocations).

\subsection{Plan of the paper}
After some preliminary results in Section~\ref{sec:prel}, we prove Theorems~\ref{th:CLT_small}, 
\ref{th:CLT_crit} and~\ref{th:CLT_large} in Sections~\ref{sec:small}, \ref{sec:crit} and~\ref{sec:large}, respectively.

\section{Preliminaries to the proofs}\label{sec:prel}
As in~\cite{MUB19}, our proofs rely on the fact that, for all $t\geq 0$, in distribution, 
$X(t) = Z(S(t))$, where $Z$ is the Markov process of semi-group $P$ started at the origin, 
and $S(t)$ is a random variable, independent of $Z$.
The idea is that, as $t\uparrow\infty$, $S(t)\sim s(t)$ in probability (where $s(t)$ is as defined in~\eqref{eq:def_s}).

To define $S(t)$, one need to introduce the following notation: 
for all $1\leq i< j$, we say that the $i$-th run is the parent of the $j$-th run if $R(j)\in [T_{i-1}, T_i)$.
This implies a genealogy on runs, and we let $i\prec j$ denote the event that the $i$-th run is an ancestor of the $j$-th run in this genealogical tree.
For any $t$, we let $i(t)$ be the integer such that $X(t)\in [T_{i(t)-1}, T_{i(t)})$, 
meaning that $X(t)$ belongs to the $i(t)$-th run of the process.
We also let $A(t)= t-T_{i(t)-1}$.
Finally, we set
\begin{equation}\label{eq:def_S}
S(t) = A(t) + \sum_{i=1}^{i(t)-1} F_i {\bf 1}_{i\prec i(t)},
\end{equation}
where, given the run-lengths $(L_i)_{i\geq 1}$, 
$(F_i)_{i\geq 1}$ is a sequence of independent random variables such that, for all $i\geq 1$,
\begin{equation}\label{eq:def_F}
\mathbb P_{\bf L}(F_i\leq x) = \frac{\int_{0}^x \mu(T_{i-1}+u) \mathrm du}{\int_{0}^{L_i} \mu(T_{i-1}+u) \mathrm du},
\end{equation}
where we let $\mathbb P_{\bf L}$ denote $\mathbb P(\,\cdot\,|(L_i)_{i\geq 1})$.
Note that, for all $i\geq 1$, $T_{i-1} + F_i$ is distributed as $R(n)$ conditioned to belong to the $i$-th run. Interestingly, and this is crucial in our proofs, the distribution of $R(n)$ conditionally on $R(n)\in [T_{i-1}, T_i)$ is the same for all $n\geq 1$.

With these definitions, we have
\begin{lemma}[See~\cite{MUB19}.]\label{lem:S(t)}
For all $t\geq 0$, in distribution,
\[X(t) = Z(S(t)),\]
where $Z$ is the Markov process of semi-group~$P$ started at the origin, and $S(t)$ is defined as in~\eqref{eq:def_S}.
\end{lemma}

The intuition behind this lemma is given in Figure~\ref{fig:S(t)}: the idea is that, by the Markov property of the underlying process, $X(t)$ can be seen at the end-point of the bold purple trajectory, which is distributed as $Z$ run for the amount of time $S(t)$.

\begin{figure}
\begin{center}
\includegraphics[width=10cm]{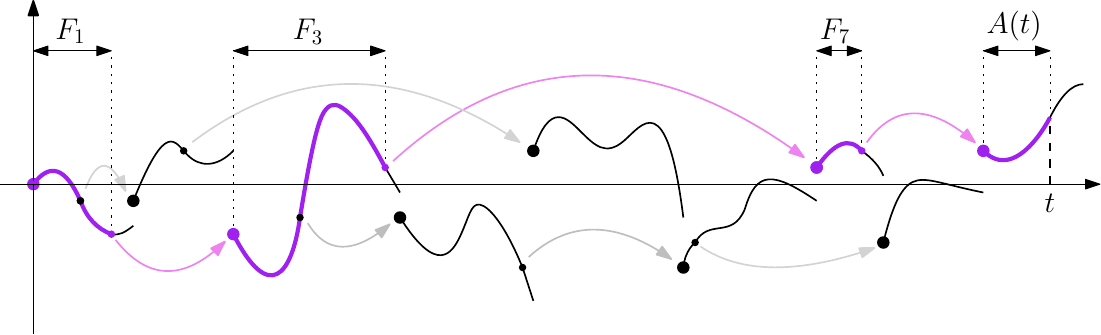}
\end{center}
\caption{Illustration of Lemma~\ref{lem:S(t)}: in this example $i(t) = 9$, and runs $1$, $3$, and $7$ are ancestors of the $9$-th run, i.e.\ $1\prec 3\prec 7\prec 9 = i(t)$.}
\label{fig:S(t)}
\end{figure}

The plan of the proof of Theorems~\ref{th:CLT_small}, \ref{th:CLT_crit}, and~\ref{th:CLT_large} all follow the same plan:
\begin{enumerate}
\item We first prove a limiting theorem for
\[\Phi(n):=\sum_{i=1}^n F_i {\bf 1}_{i\prec n}, \quad\text{ as }n\uparrow\infty.\]
\item We use renewal theory to get that $i(t) \approx t/\mathbb E[L]$ and $A(t) = o(u(t))$ in probability for any $u(t)\uparrow\infty$ as $t\uparrow\infty$. This gives a limiting theorem for $S(t)$.
\item We then compose the limiting theorem for $S(t)$ with the limiting theorem for $Z(t)$ given by Assumption~{\bf (A1)}.
\end{enumerate}
In Sections~\ref{sec:small}, \ref{sec:crit}, and \ref{sec:large}, we prove Theorems~\ref{th:CLT_small}, \ref{th:CLT_crit}, and~\ref{th:CLT_large}, respectively.
We now state a few preliminary results that will be useful in all three proofs: Lemma~\ref{lem:magic} will be useful for Step 1., while Lemmas~\ref{lem:renewal} and~\ref{lem:i(t)} will be useful in Step 2.

\begin{lemma}[{See, e.g., \cite{MUB19}}]\label{lem:magic}
Conditionally on ${\bf L} = (L_i)_{i\geq 1}$, $({\bf 1}_{i\prec n})_{1\leq i\leq n-1}$ is a sequence of independent Bernoulli-distributed random variables of respective parameters $\nicefrac{W_i}{\bar W_i}$,
where, for all $i\geq 1$,
\begin{equation}\label{eq:def_W}
W_i = \int_{T_{i-1}}^{T_i} \mu\quad
\text{ and }\quad
\bar W_i = \sum_{j=1}^i W_j = \int_0^{T_i} \mu.
\end{equation}
\end{lemma}

The proof of Lemma~\ref{lem:magic}, and the intuition behind all the results in~\cite{MUB19},
is based on the fact that the genealogical tree of the runs as defined by the relationship $\prec$ is a ``weighted (random) recursive tree'' with random weights $(W_i)_{i\geq 1}$.
For more literature on weighted recursive trees, we refer the reader to, e.g.,~\cite{BV05, BV06, D19, PS}.

\begin{lemma}\label{lem:renewal}
In probability as $t\uparrow\infty$, $A(t) = \mathcal O(1)$ (i.e., for all $\varepsilon>0$, there exists $M = M(\varepsilon)$ and $t(\varepsilon)$ such that, for all $t\geq t(\varepsilon)$, $\mathbb P(|A(t)|\geq M)\leq \varepsilon$).
In particular, for any $u : [0,\infty)\to \mathbb R$ such that $u(t)\uparrow\infty$ as $t\uparrow\infty$, $A(t)/u(t) \to 0$ in probability as $t\uparrow\infty$.
\end{lemma}

\begin{proof}
This is a straightforward consequence of the fact that, by renewal theory, $T_{i(t)+1}-T_{i(t)}$ converges in distribution to an almost surely finite random variable (see, e.g.~\cite{Bertoin}). The result follows, because $0\leq A(t)\leq T_{i(t)+1}-T_{i(t)}$ almost surely for all $t\geq 0$.
\end{proof}

\begin{lemma}\label{lem:i(t)}
Almost surely as $t\uparrow\infty$, $i(t) = \frac t{\mathbb EL} + \mathcal O(\sqrt{t\log t})$.
\end{lemma}

\begin{proof}
This is a straightforward consequence of the law of the iterated logarithm. 
First note that $i(t)\uparrow\infty$ almost surely as $t\uparrow\infty$.
Indeed, it is almost surely non-decreasing, and thus either $i(t)\uparrow\infty$, or $i(t)\uparrow i(\infty)<\infty$. In the latter case, we would get that, for all $t\geq 0$, $t\leq \sum_{i=1}^{i(\infty)} L_i$, which is impossible. Thus, $i(t)\uparrow\infty$ almost surely as claimed.
Because, by assumption, $\mathbb EL<\infty$, the strong law of large numbers gives $i(t)\sim t/\mathbb EL$, almost surely as $t\uparrow\infty$.
Now, by assumption, $\mathrm{Var}(L)<\infty$, thus,
\[-\infty<\liminf_{n\uparrow\infty} \frac{\sum_{i=1}^n L_i - n\mathbb EL}{\sqrt{n\log n}}
\leq \limsup_{n\uparrow\infty} \frac{\sum_{i=1}^n L_i - n\mathbb EL}{\sqrt{n\log n}}<\infty,\]
which implies
\[-\infty<\liminf_{t\uparrow\infty} \frac{\sum_{i=1}^{i(t)} L_i - i(t)\mathbb EL}{\sqrt{i(t)\log i(t)}}
\leq \limsup_{t\uparrow\infty} \frac{\sum_{i=1}^{i(t)} L_i - i(t)\mathbb EL}{\sqrt{i(t)\log i(t)}}<\infty.\]
Because $i(t)\sim  t/\mathbb EL$ almost surely, we get
\[-\infty<\liminf_{t\uparrow\infty} \frac{\sum_{i=1}^{i(t)} L_i - i(t)\mathbb EL}{\sqrt{t\log t}}
\leq \limsup_{t\uparrow\infty} \frac{\sum_{i=1}^{i(t)} L_i - i(t)\mathbb EL}{\sqrt{t\log t}}<\infty.\]
By definition,
\[\frac{\sum_{i=1}^{i(t)-1} L_i - i(t)\mathbb EL}{\sqrt{t\log t}}
\leq \frac{t - i(t)\mathbb EL}{\sqrt{t\log t}}
< \frac{\sum_{i=1}^{i(t)} L_i - i(t)\mathbb EL}{\sqrt{t\log t}}.\]
Taking limits, we thus get that, almost surely,
\[-\infty< \liminf_{t\uparrow\infty} \frac{t - i(t)\mathbb EL}{\sqrt{t\log t}}
\leq \limsup_{t\uparrow\infty} \frac{t - i(t)\mathbb EL}{\sqrt{t\log t}}<\infty,\]
which concludes the proof.
\end{proof}

\section{Proof of Theorem \ref{th:CLT_small}}\label{sec:small}
In this section, we assume that $\delta<1$.
As discussed in Section~\ref{sec:prel}, we start by proving a limit theorem for $\Phi(n) = \sum_{i=1}^n F_i {\bf 1}_{i\prec n}$:
\begin{proposition}\label{prop:CLT_Phi_small} 
Assume that $\delta\in (\nicefrac12, 1)$ and that the assumptions of Theorem~\ref{th:CLT_small} hold.
We let ${\bf L} = (L_i)_{i\geq 1}$ and ${\bf F} = (F_i)_{i\geq 1}$ (as defined in~\eqref{eq:def_F}).
Conditionally on $(\bold{L},\bold{F})$, $(\bold{L},\bold{F})$-almost surely, 
\begin{equation} \label{eq: CLT for Phi}
\frac{\Phi(n)-\sigma(n)}{\sqrt{\gamma n^{\delta}}}
\Rightarrow \mathcal N\left(0,\frac{\mathbb E[L^3]}{3\mathbb E[L]^{1-\delta}} \right),
\end{equation}  in distribution as $n\to \infty$, where  
\begin{equation}\label{eq:def_sigma_n}
\sigma(n) = \gamma\delta n^{\delta} \sum_{k=0}^{\lfloor \frac{\delta}{2 - 2\delta}\rfloor} 
\frac{(-\gamma\delta)^k \mathbb E[L^{k+2}]}{(k+2)!(\delta - k(1-\delta))
\mathbb E[L]^{(1-\delta)(k+1)}}\cdot  n^{-k(1-\delta)}.
\end{equation}
\end{proposition}

Before proving Proposition~\ref{prop:CLT_Phi_small} we show how to use it to prove Theorem~\ref{th:CLT_small}:
\begin{proof}[Proof of Theorem \ref{th:CLT_small}]
We first aim at proving a limit theorem for $S(t)$.
Recall that, by definition (see~\eqref{eq:def_S}),
\[S(t) = \Phi(i(t)-1) + A(t).\]
By Lemma~\ref{lem:i(t)}, $i(t)\uparrow\infty$ almost surely as $t\to \infty$. 
Thus, Proposition~\ref{prop:CLT_Phi_small} implies that, in distribution as $t\uparrow\infty$,
\begin{equation}\label{eq:small_CLT1}
\frac{\Phi(i(t)-1)-\sigma(i(t)-1)}{\sqrt{\gamma i(t)^{\delta}}}\Rightarrow 
\mathcal N\left(0,\frac{\mathbb E[L^3]}{3\mathbb E[L]^{1-\delta}} \right).
\end{equation}
Now, by Lemma~\ref{lem:i(t)}, $i(t)= t/\mathbb EL + \mathcal O(\sqrt{t\log t})$, almost surely as $t\uparrow\infty$. Thus,
\begin{equation}\label{eq:sigma(n)}
\sigma(i(t)-1) =  \gamma\delta \sum_{k=0}^{\lfloor \frac{\delta}{2 - 2\delta}\rfloor} 
\frac{(-\gamma\delta)^k \mathbb E[L^{k+2}]}{(k+2)!(\delta - k(1-\delta))
\mathbb E[L]^{(1-\delta)(k+1)}}\cdot  i(t)^{-k(1-\delta)+\delta},
\end{equation}
For all $0\leq k\leq \lfloor \frac{\delta}{2 - 2\delta}\rfloor$,
\ba
i(t)^{-k(1-\delta)+\delta} 
&= \bigg(\frac t{\mathbb EL} + \mathcal O(\sqrt{t\log t})\bigg)^{-k(1-\delta)+\delta}
= \bigg(\frac{t}{\mathbb EL}\bigg)^{-k(1-\delta)+\delta} \bigg( 1+ \mathcal O\bigg(\sqrt{\frac{\log t}{t}}\bigg)\bigg)\\
&= \bigg(\frac{t}{\mathbb EL}\bigg)^{-k(1-\delta)+\delta} + \mathcal O\big(t^{-k(1-\delta)+\delta-\nicefrac12}\sqrt{\log t}\big)
= \bigg(\frac{t}{\mathbb EL}\bigg)^{-k(1-\delta)+\delta}  + o(t^{\nicefrac\delta2}),
\ea
where, in the last equality, we have used the fact that, for all $k\geq 0$, 
$-k(1-\delta)+\delta-\nicefrac12\leq \delta-\nicefrac12<\nicefrac{\delta}2$.
Using this in~\eqref{eq:sigma(n)} gives
\[\sigma(i(t)-1)  = \gamma\delta \sum_{k=0}^{\lfloor \frac{\delta}{2 - 2\delta}\rfloor} 
\frac{(-\gamma\delta)^k \mathbb E[L^{k+2}]}{(k+2)!(\delta - k(1-\delta))
\mathbb E[L]} \cdot t^{-k(1-\delta)+\delta} + o(t^{\nicefrac\delta2})
= s(t) + o(t^{\nicefrac\delta2}),
\]
by definition of $s(t)$ is this case (see~\eqref{eq:def_s} for $\delta <1$).
Using this and the fact that $i(t)\sim t/\mathbb EL$ in~\eqref{eq:small_CLT1}, we get
\[\frac{\Phi(i(t))-s(t) + o(t^{\nicefrac\delta2})}{\sqrt{\gamma (t/\mathbb EL)^{\delta}}}\Rightarrow 
\mathcal N\left(0,\frac{\mathbb E[L^3]}{3\mathbb E[L]^{1-\delta}} \right),\]
which implies
\[\frac{\Phi(i(t)-1)-s(t)}{t^{\nicefrac\delta2}}\Rightarrow 
\mathcal N\left(0,\frac{\gamma\mathbb E[L^3]}{3\mathbb E[L]} \right).\]
Recall that $S(t) = \Phi(i(t)-1) + A(t)$ (see~\eqref{eq:def_S});
using the fact that, by Lemma~\ref{lem:renewal}, $A(t)/t^{\nicefrac\delta2}\to 0$ in probability, we get that, in distribution as $t\uparrow\infty$.
\begin{equation}\label{eq:small_CLT_S}
\frac{S(t)-s(t)}{t^{\nicefrac\delta2}}\Rightarrow 
\mathcal N\left(0,\frac{\gamma\mathbb E[L^3]}{3\mathbb E[L]}\right).
\end{equation}

By Skorokhod's representation theorem, 
there exists a probability space on which there exist $\tilde{S}(t)$ distributed as $S(t)$ for all $t\geq 0$, and $\hat\Omega\sim \mathcal N(0,{\mathbb E[L^3]}/(3\mathbb E[L]))$, 
such that almost surely when $t\to \infty$,
\ba
\tilde{S}(t) 
& = s(t) + \hat\Omega \sqrt{\gamma t^{ \delta }} + o(t^{\nicefrac\delta2}) \\
& = s(t)+ \hat\Omega \sqrt{ \frac{2 \mathbb E[L] }{ \mathbb E[L^2] } s(t) } 
+ \hat\Omega\left(\sqrt{\gamma t^{\delta}} - \sqrt{ \frac{2 \mathbb E[L] }{ \mathbb E[L^2] } s(t) }\right) + o(t^{\nicefrac\delta2}) \\
& = s(t) + \hat\Omega \sqrt{ \frac{2 \mathbb E[L] }{ \mathbb E[L^2] }s(t)} + o\left(\sqrt{s(t)}\right).
\ea
Indeed, we have used the fact that, as $t\uparrow\infty$, 
$s(t)\sim {\gamma \mathbb E[L^2]t^{\delta}}/(2\mathbb E [L])$ (see Remark~\ref{rk:123}).
By Assumption {\bf (A1)}, on the same probability space, 
there exists $(\tilde{Z}(t))_{t\geq 0}$ independent of $(\tilde S(t))_{t\geq 0}$ and a random variable $\Lambda$, independent of $\hat\Omega$ such that
$\tilde{Z}(t)$ has the same distribution as $Z(t)$ for all $t\geq 0$, $\Lambda$ has distribution~$\nu$, and ${\tilde{Z}(t)-a(t)}/{b(t)}\to \Lambda$ almost surely as $t\uparrow\infty$. 
Therefore, by assumption \textbf{(A2)}, almost surely as $t\to \infty$, 
\ba
\frac{\tilde{Z}(\tilde{S}(t))-a(s(t))}{b(s(t))} 
& = \frac{b(S(t))}{b(s(t))} \cdot \frac{\tilde{Z}(\tilde{S}(t))-a(\tilde S(t))}{b(\tilde{S}(t))} 
+\frac{a(\tilde{S}(t))-a(s(t))}{b(s(t))} \\
&\to f(\Omega)+ \Lambda g(\Omega),
\ea
where $\Omega = \hat\Omega \sqrt{{2\mathbb E[L]}/{\mathbb E[L^2]}}\sim \mathcal N(0, {2\mathbb E[L^3]}/(3\mathbb E[L^2]))$.
Because $\tilde Z(\tilde S(t)) = Z(S(t)) = X(t)$ in distribution for all $t\geq 0$ (by Lemma~\ref{lem:S(t)}), 
this concludes the proof.
\end{proof}

%\subsection{Proof of Proposition \ref{CLT Phi}} 
The rest of the section is devoted to proving Proposition~\ref{prop:CLT_Phi_small}. 
The idea of the proof is as follows: we first reason conditionally on ${\bf L}$ and ${\bf F}$ and 
apply Linderberg's central limit theorem to prove that
\begin{equation}\label{eq:Lind}
\frac{\Phi(n)-\mathbb E_{{\bf L}, {\bf F}}[\Phi(n)]}{\sqrt{\mathrm{Var}_{{\bf L}, {\bf F}}(\Phi(n))}}\Rightarrow \Nor(0,1)
\end{equation}
in distribution as $n\uparrow\infty$.
This is done in Section~\ref{sec:Lindeberg}.
To prove Proposition~\ref{prop:CLT_Phi_small}, we thus need to find asymptotic equivalents 
for $\mathbb E_{{\bf L}, {\bf F}}[\Phi(n)]$ and $\mathrm{Var}_{{\bf L}, {\bf F}}(\Phi(n))$.
To do so, we apply Lemma~\ref{lem:magic} to get that
\[\mathbb E_{{\bf L}, {\bf F}}[\Phi(n)]
=\sum_{i=1}^{n} \frac{F_iW_i}{\bar W_i}\quad \text{and}\quad 
\mathrm{Var}_{{\bf L}, {\bf F}}(\Phi(n))
=\sum_{i=1}^{n} \frac{F_iW_i}{\bar W_i}\left(1-\frac{W_i}{\bar W_i}\right).\] 
In Lemma~\ref{expectation and variance approximation}, 
we prove that, conditionally on ${\bf L}$, 
we can almost surely approximate $\mathbb E_{{\bf L}, {\bf F}}[\Phi(n)]$ and $\mathrm{Var}_{{\bf L}, {\bf F}}(\Phi(n))$ by their expectations conditionally on ${\bf L}$, i.e.,
\begin{equation}\label{eq:approx_by_exp}
\mathbb E_{{\bf L}, {\bf F}}[\Phi(n)]\approx \sum_{i=1}^{n} \frac{\mathbb E_{\bf L}[F_i]W_i}{\bar W_i}
\quad\text{ and }\quad
\mathrm{Var}_{{\bf L}, {\bf F}}(\Phi(n))\approx \sum_{i=1}^{n} \frac{\mathbb E_{\bf L}[F_i]W_i}{\bar W_i}
\left(1-\frac{W_i}{\bar W_i}\right).
\end{equation}
To estimate these expectations, we note that, by Definition of $(F_i)_{i\geq 1}$ (see~\eqref{eq:def_F}, see also~\eqref{eq:def_W} for the definition of $(W_i)_{i\geq 1}$),
\begin{equation}\label{eq:exp_F}
\mathbb E_{\bf L}[F_i] 
= \frac1{W_i}\int_0^{L_i} x\mu(T_{i-1} + x)\mathrm dx
= \frac{1}{W_i}\int_0^{L_i} \gamma\delta x(T_{i-1}+x)^{\delta-1} \mathrm e^{\gamma (T_{i-1}+x)^\delta}\mathrm dx.
\end{equation}
In Section~\ref{sub:small_prel}, we prove preliminary technical results that will be useful in the rest.
In Section~\ref{sub:approx_by_exp} we state and prove a rigorous version of~\eqref{eq:approx_by_exp}.
Finally, in Section~\ref{sec:Lindeberg}, we use Linderberg's theorem to prove~\eqref{eq:Lind} and conclude the proof of Proposition~\ref{prop:CLT_Phi_small}.

\subsection{A preliminary lemma}\label{sub:small_prel}
In this section, we prove the following lemma: recall that, for all $i\geq 1$, $T_i = \sum_{j=1}^i L_i$ is the time of the $i$-th relocation.
\begin{lemma}\label{lem:sum t_i} 
Let $g:\RR \to \RR$ be a function such that $\mathbb E[g(L)],\mathrm{Var}(g(L))<\infty$. 
\begin{itemize}
\item[{\rm\bf(i)}] For all $\ell\in(0,1)$ almost surely as $n\to \infty$,  
\[\sum_{i=1}^n \frac{g(L_i)}{T_i^{1-\ell}}
= \frac{\mathbb E[g(L)]}{\ell \cdot \mathbb E[L]^{1-\ell}} n^{\ell} + o(n^{\ell/2}).\]
\item[{\rm\bf(ii)}] For all $\ell>1$, almost surely as $n\to \infty$,
 \[\sum_{i=1}^n \frac{g(L_i)}{T_i^\ell} 
= \mathcal O(1).\]
\item[{\rm\bf(iii)}]
Almost surely as $n\uparrow\infty$,
\[\sum_{i=1}^n \frac{g(L_i)}{T_i} = \mathbb E[g(L)]\log n + \mathcal O(1).\]
\end{itemize}
\end{lemma}

To prove this lemma, we use the following strong law of large numbers, for sums of independent but not identically-distributed random variables:
\begin{theorem}[(see~\cite{petrov})]\label{th:SLLN_Petrov}
Let $(\Delta_i)_{i\geq 1}$ be a sequence of independent random variables,
and let $(a_i)_{i\geq 1}$ be a sequence of real numbers such that $a_n\to+\infty$ as $n\uparrow\infty$.
Assume that there exists $\alpha\in[1,2]$ such that
\[\sum_{i\geq 1} \frac{\mathbb E[|\Delta_i|^\alpha]}{a_i^\alpha}<\infty.\]
Then, almost surely as $n\uparrow\infty$,
\[\frac1{a_n}\sum_{i=1}^n \Delta_i \to 0.\]
\end{theorem}

\begin{proof}[Proof of Lemma~\ref{lem:sum t_i}] 
{\bf (i)} By the law of the iterated logarithm $T_i=\mathbb E[L] i + \mathcal{O}(\sqrt{i \log i})$ almost surely as $i\to \infty$. Hence,  
\ban
\sum_{i=1}^n g(L_i)T_{i}^{\ell -1} 
& = \sum_{i=1}^n g(L_i)\left( \mathbb E[L] i 
+ \mathcal{O}(\sqrt{i \log i})\right)^{\ell -1} \nonumber \\
& = \mathbb E[L]^{\ell -1} \sum_{i=1}^n g(L_i) i^{\ell -1} 
+ \mathcal{O}\left( \sum_{i=1}^n g(L_i) i^{\ell -\nicefrac32} \sqrt{\log i} \right), \label{sum t_i}
\ean
almost surely as $n\uparrow\infty$.
For the first sum in \eqref{sum t_i}, note that
\begin{equation}\label{eq:small_bla}
\sum_{i=1}^n g(L_i) i^{\ell -1} =
\sum_{i=1}^n \mathbb E[g(L_i)] i^{\ell -1}+
\sum_{i=1}^n \big(g(L_i)-\mathbb E[g(L_i)]\big) i^{\ell -1}.
\end{equation}
The latter sum is a sum of independent random variables to which we can apply 
Theorem~\ref{th:SLLN_Petrov} with $\alpha=2$ and $a_i = i^{\nicefrac\ell2}$. 
Indeed,
\[\sum_{i=1}^{\infty} \mathbb E\big[(g(L_i)-\mathbb E[g(L_i)])^2\big]i^{2(\ell-1)-\ell} 
= \mathrm{Var}(g(L)) \sum_{i=1}^{\infty} \frac1{i^{\ell-2}}
<\infty,
\]
because $\ell-2<-1$, by assumption on $\ell$.
Thus, Theorem~\ref{th:SLLN_Petrov} implies that, almost surely as $n\uparrow\infty$
\[\sum_{i=1}^n \big(g(L_i)-\mathbb E[g(L_i)]\big) i^{\ell -1} = o(n^{\nicefrac\ell2}),\]
and thus, by~\eqref{eq:small_bla},
\[\sum_{i=1}^n g(L_i) i^{\ell -1} =
\sum_{i=1}^n \mathbb E[g(L_i)] i^{\ell -1}+o(n^{\nicefrac\ell2})
= \mathbb E[g(L)] \sum_{i=1}^n i^{\ell -1} + o(n^{\nicefrac\ell2}).\]
Now note that
\[\frac{(n+1)^{\ell}-1}\ell = \int_1^{n+1} x^{\ell-1}\mathrm dx\leq \sum_{i=1}^n i^{\ell -1}
\leq \int_0^n x^{\ell-1}\mathrm dx = \frac{n^\ell}{\ell},
\]
which implies that, as $n\uparrow\infty$,
\[\sum_{i=1}^n i^{\ell -1}  = \frac{n^{\ell}}{\ell} +o(n^{\nicefrac\ell2}).\]
We thus get that, almost surely as $n\uparrow\infty$,
\begin{equation}\label{eq:small_g(L)}
\sum_{i=1}^n g(L_i)i^{\ell -1} = \frac{\mathbb E[g(L)]}{\ell}\cdot n^{\ell} + o(n^{\nicefrac\ell2}).
\end{equation} 
For the second sum in~\eqref{sum t_i}, 
we proceed similarly: first note that, for all $\varepsilon>0$,
\[\sum_{i=1}^n g(L_i) i^{\ell -3/2} \sqrt{\log i} =
\mathcal O\bigg(\sum_{i=1}^n g(L_i) i^{\ell -3/2+\varepsilon}\bigg).\]
We use Theorem~\ref{th:SLLN_Petrov} again, with $\alpha = 2$ and $a_n = n^{\varepsilon}$.
Indeed, we have
\[\sum_{i\geq 1} \big(g(L_i)-\mathbb E[g(L_i)]\big) i^{2\ell -3+2\varepsilon-2\varepsilon}
= \mathrm{Var}(g(L)) \sum_{i\geq 1} i^{2\ell -3}<\infty,
\]
because $2\ell-3<-1$ by assumption on $\ell$.
Theorem~\ref{th:SLLN_Petrov} thus implies that
\[\sum_{i=1}^n g(L_i) i^{\ell -3/2+\varepsilon} 
= \mathbb E[g(L)] \sum_{i=1}^n i^{\ell - 3/2+\varepsilon} + o(n^{2\varepsilon}).
\]
We now choose $\varepsilon$ such that $0<\varepsilon<\min(\ell-\frac12, \frac\ell4, \frac{(1-\ell)}2)$.
Using the fact that
\[0\leq \sum_{i=1}^n i^{\ell -3/2+\varepsilon}\leq \int_1^{n+1} x^{\ell -3/2+\varepsilon}\mathrm dx
= \mathcal O(n^{\ell-\nicefrac12+\varepsilon}),
\]
we get that 
\[\sum_{i=1}^n g(L_i) i^{\ell -3/2+\varepsilon} 
= \mathcal O(n^{\ell-\nicefrac12+\varepsilon}) + o(n^{2\varepsilon})
= o(n^{\nicefrac\ell2}),\]
because, with our choice of $\varepsilon$, $\ell-\nicefrac12+\varepsilon<\nicefrac\ell2$ and $2\varepsilon<\nicefrac\ell2$.
We have thus proved that, almost surely as $n\uparrow\infty$,
\[\sum_{i=1}^n g(L_i) i^{\ell -3/2} \sqrt{\log i} = o(n^{\nicefrac\ell2}).\]
Together with~\eqref{sum t_i} and~\eqref{eq:small_g(L)}, this gives that, almost surely as $n\uparrow\infty$,
\[\sum_{i=1}^n g(L_i)T_{i}^{\ell -1} 
= \frac{\mathbb E[g(L)]}{\ell}\cdot n^{\ell} + o(n^{\nicefrac\ell2}),\]
which concludes the proof of the first claim.

{\bf (iii)} We use the strong law of large numbers,
to get that $T_i\sim i\mathbb E[L]$ almost surely as $i\uparrow\infty$.
This implies
\[\sum_{i=1}^n \frac{g(L_i)}{T_{i}^{\ell}} 
 = \mathcal O\bigg(\frac1{\mathbb E[L]^{\ell}} \sum_{i=1}^n \frac{g(L_i)}{i^{\ell}}\bigg).\]
Now, we write
\[ \sum_{i=1}^n \frac{g(L_i)}{i^{\ell}}
= \mathbb E[g(L)] \sum_{i=1}^n  \frac1{i^{\ell}} + \sum_{i=1}^n \frac{g(L_i)-\mathbb E[g(L)]}{i^{\ell}},
\]
and note that the second sum is a martingale whose quadratic variation satisfies
\[\sum_{i=1}^n \frac{\mathbb E\big[(g(L_i)-\mathbb E[g(L)])^2\big]}{i^{2\ell}} 
= \mathrm{Var}(g(L)) \sum_{i=1}^n \frac1{i^{2\ell}}
\leq \mathrm{Var}(g(L)) \sum_{i=1}^\infty \frac1{i^{2\ell}}<\infty,
\]
because $\ell>1$.
Thus, this martingale converges almost surely as $n\uparrow\infty$, and we get
\[\sum_{i=1}^n \frac{g(L_i)}{ i^{\ell}} 
= \mathbb E[g(L)] \sum_{i=1}^n  \frac1{i^{\ell}} + \mathcal O(1)
= \mathcal O(1),\]
because $\ell>1$. 
This concludes the proof of the second claim.

{\bf (iii)} We proceed similarly to the proof of {\bf (i)} and write that, by the law of the iterated logarithm,
\ba
\sum_{i=1}^n \frac{g(L_i)}{T_i}
&= \sum_{i=1}^n \frac{g(L_i)}{(i\mathbb EL + \mathcal O(\sqrt{i\log i}))}
= \sum_{i=1}^n \frac{g(L_i)}{i\mathbb EL (1 + \mathcal O(\sqrt{(\log i)/i}))}\\
&= \sum_{i=1}^n \frac{g(L_i)}{i\mathbb EL} + \mathcal O\bigg(\sum_{i=1}^n \frac{g(L_i)\sqrt{\log i}}{i^{\nicefrac32}\mathbb EL}\bigg).
\ea
We now use the fact that $\log i$ is negligible in front of any power or $i$ to get that
\[\sum_{i=1}^n \frac{g(L_i)\sqrt{\log i}}{i^{\nicefrac32}\mathbb EL}
= \mathcal O\bigg(\sum_{i=1}^n \frac{g(L_i)}{i^{\nicefrac54}\mathbb EL}\bigg)
= \mathcal O(1),\]
by {\bf (ii)}.
This implies that, almost surely as $n\uparrow\infty$,
\ba
\sum_{i=1}^n \frac{g(L_i)}{T_i}
&= \sum_{i=1}^n \frac{g(L_i)}{i\mathbb EL}  + \mathcal O(1)
= \mathbb E[g(L)]\sum_{i=1}^n \frac{1}{i\mathbb EL} + \sum_{i=1}^n \frac{g(L_i)-\mathbb E[g(L)]}{i\mathbb EL} +\mathcal O(1)\\
&= \frac{\mathbb E[g(L)]}{\mathbb EL}\cdot \log n + \sum_{i=1}^n \frac{g(L_i)-\mathbb E[g(L)]}{i\mathbb EL} +\mathcal O(1).
\ea
Now, the remaining sum is a martingale whose quadratic variation satisfies
\[\sum_{i=1}^n \frac{\mathbb E[(g(L_i)-\mathbb E[g(L)])^1]}{i^2(\mathbb EL)^2}
= \frac{\mathrm{Var}(g(L))}{(\mathbb EL)^2} \sum_{i=1}^n \frac1{i^2}\leq
 \frac{\mathrm{Var}(g(L))}{(\mathbb EL)^2} \sum_{i\geq 1} \frac1{i^2}<\infty.
\]
This implies that the remaining sum converges almost surely as $n\uparrow\infty$, which concludes the proof.
\end{proof}

\subsection{Estimating the sums in~\eqref{eq:approx_by_exp}}\label{sub:small_sums}
The aim of this section is to understand the asymptotic behaviour of $\sum_{i=1}^n {W_i\mathbb E_{\bs L}[F_i]}/{\bar W_i}$ and $\sum_{i=1}^n {W_i\mathbb E_{\bs L}[F^2_i]}/{\bar W_i}$. We do this in two separate lemmas.

\begin{lemma}\label{lem:small_expectation}
Almost surely as $n\to+\infty$,
\[\sum_{i=1}^n \frac{W_i\mathbb E_{\bs L}[F_i]}{\bar W_i}
= \sigma(n) + o(n^{\nicefrac\delta2}).\] 
\end{lemma}

%\begin{remark} 
%The standard definition of $\mathcal{O}(\cdot)$ is the following: 
%$u_i = \mathcal O(v_i)$ if there exists a constant $C>0$ and an integer $i_0$ such that, 
%for all $i\geq i_0$, $|u_i|\leq C v_i$.
%In Lemma \ref{lem: summands approx}, 
%we allow~$i_0$ to be a random variable that depends on ${\bf L}$, 
%but we assume that~$C$ a deterministic constant.
%\end{remark}

To prove Lemma~\ref{lem:small_expectation}, 
we need the following consequence of Borel-Cantelli lemma:
\begin{lemma}\label{lem:small_BC}
For all $\delta\in (0,1)$, $L_i/T_{i-1}^{1-\delta}\to 0$ almost surely as $i\uparrow\infty$.
\end{lemma}

\begin{proof}
First note that, by the strong law of large numbers, 
almost surely for all $i$ large enough,
$T_{i-1}\geq i\mathbb EL/2$.
We take $p$ as in~\eqref{eq:def_p} (recall that, in this section, we assume that $\delta\in(0,1)$).
Fix $\varepsilon>0$;
by Markov inequality, for all $i\geq 1$,
\[\mathbb P(L_i>\varepsilon (2i\mathbb EL)^{1-\delta})
\leq \frac{\mathbb E[L^p]}{\varepsilon^p(i\mathbb EL/2)^{(1-\delta)p}}.\]
By definition of $p$ (see~\eqref{eq:def_p}), $(1-\delta)p>1$.
and thus, $\sum_{i\geq 1} \mathbb P(L_i>\varepsilon(2i\mathbb EL)^{1-\ell})<\infty$.
By Borel-Cantelli lemma (because the $(L_i)_{i\geq 1}$ is a sequence of independent random variables), almost surely for all $i$ large enough, $L_i\leq \varepsilon (2i\mathbb EL)^{1-\ell}$,
which concludes the proof (recall that $T_{i-1}\geq i\mathbb EL/2$ almost surely for all $i$ enough).
\end{proof}

\begin{proof}[Proof of Lemma~\ref{lem:small_expectation}]
By~\eqref{eq:exp_F}, for all $i\geq 1$,
\ba
W_i\mathbb E_{\bs L} [F_i]
&= \int_0^{L_i} \gamma\delta x(T_{i-1}+x)^{\delta-1} \mathrm e^{\gamma (T_{i-1}+x)^\delta}\mathrm dx
= L_i \mathrm e^{\gamma T_i^\delta}
- \int_0^{L_i} \mathrm e^{\gamma (T_{i-1}+x)^\delta}\mathrm dx,
\ea
by integration by parts.
Setting $u = (T_{i-1}+x)^\delta$, we get
\ba
W_i\mathbb E_{\bs L} [F_i]
&= L_i \mathrm e^{\gamma T_i^\delta}
- \int_{T_{i-1}^\delta}^{T^{\delta}_i} \frac1\delta u^{\frac1\delta-1} \mathrm e^{\gamma u}\mathrm du
= L_i \mathrm e^{\gamma T_i^\delta}
- \frac{T_i^{1-\delta}}{\delta}\int_{T_{i-1}^\delta}^{T^{\delta}_i} \mathrm e^{\gamma u}\mathrm du
+  \frac1\delta\int_{T_{i-1}^\delta}^{T^{\delta}_i}  \big[T_i^{1-\delta} -u^{\frac1\delta-1}\big] \mathrm e^{\gamma u}\mathrm du\\
&= L_i \mathrm e^{\gamma T_i^\delta}
- \frac{T_i^{1-\delta}}{\gamma\delta}\big(\mathrm e^{\gamma T_i^{\delta}}-\mathrm e^{\gamma T_{i-1}^{\delta}}\big)
+  \frac1\delta\int_{T_{i-1}^\delta}^{T^{\delta}_i}  \big[T_i^{1-\delta} -u^{\frac1\delta-1}\big] \mathrm e^{\gamma u}\mathrm du
\ea
thus,
\begin{equation}\label{eq:small_first_appr}
\sum_{i=1}^n\frac{W_i\mathbb E_{\bs L} [F_i]}{\bar W_i}
= T_n  
- \sum_{i=1}^n \frac{T_i^{1-\delta}}{\gamma\delta}\big(1-\mathrm e^{-\gamma (T_i^{\delta}-T_{i-1}^{\delta})}\big) 
+ R_n,
\end{equation}
where we have set
\[R_n = \sum_{i=1}^n \frac1\delta\int_{T_{i-1}^\delta}^{T^{\delta}_i}  \big[T_i^{1-\delta} -u^{\frac1\delta-1}\big] \mathrm e^{\gamma u}\mathrm du.\]
We first show that, in probability as $n\uparrow\infty$,
$R_n = o(n^{\nicefrac\delta2})$. 
Indeed, we have $R_n\geq 0$, and
\ba
R_n &\leq \sum_{i=1}^n \big(T_i^{1-\delta} -T_{i-1}^{1-\delta}\big)
\sum_{i=1}^n \frac1\delta\int_{T_{i-1}^\delta}^{T^{\delta}_i} \mathrm e^{-\gamma (T_i^\delta - u)}\mathrm du
= \sum_{i=1}^n T_i^{1-\delta}\Big(1-\Big(1-\frac{L_i}{T_i}\Big)^{1-\delta}\Big)
\big(1-\mathrm e^{-\gamma (T_i^\delta -T_{i-1}^{\delta})}\big)\\
& =\sum_{i=1}^n T_i^{1-\delta}\Big(1-\Big(1-\frac{L_i}{T_i}\Big)^{1-\delta}\Big)
\big(1-\mathrm e^{-\gamma T_i^{\delta} (1 -(1-L_i/T_i)^{\delta})}\big)
\ea
As $i\uparrow\infty$, $L_i/T_i \to 0$ almost surely (see Lemma~\ref{lem:small_BC}), 
and thus
\[T_i^{1-\delta}\Big(1-\Big(1-\frac{L_i}{T_i}\Big)^{1-\delta}\Big)
\big(1-\mathrm e^{-\gamma T_i^{\delta} (1 -(1-L_i/T_i)^{\delta})}\big)
\sim (1-\delta)L_iT_i^{-\delta} \big(1-\mathrm e^{-\gamma\delta L_i T_i^{\delta-1}}\big)
\sim \frac{\gamma\delta(1-\delta)L^2_i}{T_i},
\]
because $\delta<1$ and thus $L_iT_i^{\delta -1} \to 0$ almost surely as $i\uparrow\infty$ (by Lemma~\ref{lem:small_BC}).
Because, by Lemma~\ref{lem:sum t_i},
\[\sum_{i=1}^n \frac{\gamma\delta(1-\delta)L^2_i}{T_i} = \mathcal O(\log n),\]
we get that, almost surely as $n\uparrow\infty$
\[R_n = \mathcal O(\log n) = o(n^{\nicefrac\delta2}),\]
as needed.
We now look at the second term in~\eqref{eq:small_first_appr}: we write
\[\sum_{i=1}^n \frac{T_i^{1-\delta}}{\gamma\delta}\big(1-\mathrm e^{-\gamma (T_i^{\delta}-T_{i-1}^{\delta})}\big) 
= \sum_{i=1}^n \frac{T_i^{1-\delta}}{\gamma\delta}\big(1-\mathrm e^{-\gamma\delta L_i T_i^{\delta-1}}\big)
+ R'_n,\]
where
\[R'_n = \sum_{i=1}^n \frac{T_i^{1-\delta}}{\gamma\delta}\big(\mathrm e^{-\gamma\delta L_i T_i^{\delta-1}}-\mathrm e^{-\gamma (T_i^{\delta}-T_{i-1}^{\delta})}\big).\]
First note that, for all $x\in [0,1]$, $(1-x)^\delta \leq 1-\delta x$. This implies
\[T_i^{\delta}-T_{i-1}^{\delta} 
= T_i^{\delta} \Big(1-\Big(1-\frac{L_i}{T_i}\Big)^{\delta}\Big)
\geq \delta L_i T_i^{\delta-1},\]
implying that $R'_n\geq 0$ for all $n\geq 1$.
We now prove that $R'_n = o(n^{\nicefrac\delta2})$ almost surely as $n\uparrow\infty$.
To do so, we use the fact that, for all $x\in[0,1]$, $(1-x)^\delta \leq 1-\delta x-\delta(1-\delta)x^2/2$, and thus
\[T_i^{\delta}-T_{i-1}^{\delta} \leq \delta L_i T_i^{\delta-1} + \delta(1-\delta)L_i^2 T_i^{\delta-2}/2,
\]
which, in turn, implies
\ba
R'_n
&\leq \sum_{i=1}^n \frac{T_i^{1-\delta}}{\gamma\delta}\mathrm e^{-\gamma\delta L_i T_i^{\delta-1}}\big(1-\mathrm e^{-\gamma\delta(1-\delta)L_i^2 T_i^{\delta-2}/2}\big)
\leq \sum_{i=1}^n \frac{T_i^{1-\delta}}{\gamma\delta}
\big(1-\mathrm e^{-\gamma\delta(1-\delta)L_i^2 T_i^{\delta-2}/2}\big)\\
&\leq \sum_{i=1}^n \frac{(1-\delta) L_i^2}{2 T_i} = \mathcal O(\log n),
\ea
by Lemma~\ref{lem:sum t_i}.
In total, using the fact that $R_n = o(n^{\nicefrac\delta2})$ and $R'_n = o(n^{\nicefrac\delta2})$ in \eqref{eq:small_first_appr}, we get that, almost surely as $n\uparrow\infty$,
\ba\sum_{i=1}^n\frac{W_i\mathbb E_{\bs L} [F_i]}{\bar W_i}
&= T_n  
- \sum_{i=1}^n \frac{T_i^{1-\delta}}{\gamma\delta}\big(1-\mathrm e^{-\gamma\delta L_i T_i^{\delta-1}}\big) + o(n^{\nicefrac\delta2})\\
&=  \sum_{i=1}^n \frac{T_i^{1-\delta}}{\gamma\delta}\big(\mathrm e^{-\gamma\delta L_i T_i^{\delta-1}}-1+\gamma\delta L_i T_i^{\delta-1}\big) + o(n^{\nicefrac\delta2})\\
&= \sum_{i=1}^n \frac{T_i^{1-\delta}}{\gamma\delta}\sum_{k\geq 2}\frac{(-\gamma\delta L_i T_i^{\delta-1})^k}{k!}+ o(n^{\nicefrac\delta2})
= \sum_{k\geq 2}\frac{(-\gamma\delta)^k}{\gamma\delta k!} \sum_{i=1}^n \frac{L_i^k}{T_i^{(1-\delta)(k-1)}}+ o(n^{\nicefrac\delta2})
\ea
Now note that, for all $k\geq k_0 = \lfloor\frac{\nicefrac32-\delta}{1-\delta}\rfloor+1$, 
\[\sum_{i=1}^n \frac{L_i^k}{T_i^{(1-\delta)(k-1)}}
\leq \sum_{i=1}^n \frac{L_i^{k_0}}{T_i^{(1-\delta)(k_0-1)}}=o(n^{\nicefrac\delta2}),\]
almost surely as $n\uparrow\infty$.
Furthermore, because of the alternating signs,
\[\left|\sum_{k\geq k_0}\frac{(-\gamma\delta)^k}{\gamma\delta k!} \sum_{i=1}^n \frac{L_i^k}{T_i^{(1-\delta)(k-1)}}\right|
\leq \left|\frac{(-\gamma\delta)^{k_0}}{\gamma\delta k_0!} \sum_{i=1}^n \frac{L_i^{k_0}}{T_i^{(1-\delta)(k_0-1)}}\right| =o(n^{\nicefrac\delta2}),\]
almost surely.
This gives
\ba
\sum_{i=1}^n\frac{W_i\mathbb E_{\bs L} [F_i]}{\bar W_i}
&= \sum_{k= 2}^{k_0-1}\frac{(-\gamma\delta)^k}{\gamma\delta k!} \sum_{i=1}^n \frac{L_i^k}{T_i^{(1-\delta)(k-1)}}+ o(n^{\nicefrac\delta2})\\
&= \sum_{k= 2}^{k_0-1}\frac{(-\gamma\delta)^k\mathbb E[L^k]}{\gamma\delta \mathbb E[L]^{(1-\delta)(k-1)}k!} \cdot\frac{n^{(1-\delta)(k-1)+1}}{(1-\delta)(k-1)+1}+ o(n^{\nicefrac\delta2}),
\ea
which concludes the proof, by definition of $\sigma(n)$ (see~\eqref{eq:def_sigma_n}).
\end{proof}

\begin{lemma}\label{lem:small_sec_moment} 
Almost surely as $n\to \infty$,
\[\sum_{i=1}^n \frac{\mathbb E_{\bold{L}}[F_i^2] W_i}{\bar W_i} 
= \frac{\gamma\mathbb E[L^3]}{3\mathbb E[L]^{1-\delta}} \cdot n^{\delta} +o(n^{\delta}).\]
\end{lemma}

\begin{proof}
By~\eqref{eq:exp_F}, for all $i\geq 1$,
\ba
W_i\mathbb E_{\bs L} [F_i^2]
&= \int_0^{L_i} \gamma\delta x^2(T_{i-1}+x)^{\delta-1} \mathrm e^{\gamma (T_{i-1}+x)^\delta}\mathrm dx
= L^2_i \mathrm e^{\gamma T_i^\delta}
- \int_0^{L_i} 2x\mathrm e^{\gamma (T_{i-1}+x)^\delta}\mathrm dx,
\ea
by integration by parts.
This implies
\ba\frac{W_i\mathbb E_{\bs L} [F_i^2]}{\bar W_i}
&= L^2_i - \int_0^{L_i} 2x\mathrm e^{-\gamma (T_i^\delta - (T_{i-1}+x)^\delta)}\mathrm dx
= L^2_i - \int_0^{L_i} 2x\mathrm e^{-\gamma T_i^\delta (1 - (1-(L_i-x)/T_i)^\delta)}\mathrm dx\\
&= L^2_i - \int_0^{L_i} 2(L_i-u)\mathrm e^{-\gamma T_i^\delta (1 - (1-\nicefrac u{T_i})^\delta)}\mathrm du
= L^2_i - \int_0^{L_i} 2(L_i-u)(1-\gamma\delta u T_i^{\delta-1})\mathrm du
+ r_i,
\ea
where we have set 
\[r_i = \int_0^{L_i} 2(L_i-u)\big(1-\gamma\delta u T_i^{\delta-1}-\mathrm e^{-\gamma T_i^\delta (1 - (1-\nicefrac u{T_i})^\delta)}\big)\mathrm du.\]
Thus,
\[\frac{W_i\mathbb E_{\bs L} [F_i^2]}{\bar W_i}
= \gamma\delta T_i^{\delta-1} \int_0^{L_i} 2(L_i-u) u\mathrm du
+ r_i
= \frac{\gamma\delta L_i^3}{3T_i^{1-\delta}} + r_i,\]
which, by Lemma~\ref{lem:sum t_i}, implies
\[\sum_{i=1}^n \frac{W_i\mathbb E_{\bs L} [F_i^2]}{\bar W_i}
= \frac{\gamma \mathbb E[L^3]}{2(\mathbb EL)^{1-\delta}}\cdot n^\delta + o(n^{\nicefrac\delta2}) + \sum_{i=1}^n r_i.\]
Thus, it only remains to prove that $\sum_{i=1}^n r_i = o(n^\delta)$ almost surely as $n\uparrow\infty$.
First note that, almost surely for all $i\geq 1$, $\varpi_i : u\mapsto \mathrm e^{-\gamma T_i^\delta (1 - (1-\nicefrac u{T_i})^\delta)}+\gamma\delta u T_i^{\delta-1}-1$ is non-decreasing on $[0,L_i]$. 
Indeed, we have, uniformly in $u\in [0,L_i]$, because $L_i/T_i\to 0$ almost surely as $i\uparrow\infty$ (see Lemma~\ref{lem:small_BC}),
\ba
\varpi'_i(u) 
&= \gamma\delta T_i^{\delta-1}\Big(1-\Big(1-\frac u{T_i}\Big)^{\delta-1}\mathrm e^{-\gamma T_i^\delta (1 - (1-\nicefrac u{T_i})^\delta)}\Big)
\sim \gamma\delta T_i^{\delta-1}\Big(1-\Big(1+\frac{(1-\delta)u}{T_i}\Big)\Big(1-\frac{\gamma\delta u}{T_i^{}1-\delta}\Big)\\
&\sim (\gamma\delta)^2 u,
\ea
almost surely as $i\uparrow\infty$.
Thus, almost surely for all $i$ large enough, $\varpi_i$ is non-decreasing on $[0, L_i]$ as claimed.
Because $\varpi_i(0) = 0$, this implies that $0\leq \varpi_i(u)\leq \varpi_i(L_i)$ for all $u\in [0,L_i]$.
This implies that
\[|r_i| \leq \varpi_i(L_i) \int_0^{L_i} 2(L_i-u)\mathrm du
= L_i^2 \varpi_i(L_i)
= L_i^2 \mathrm e^{-\gamma T_i^\delta (1 - (1-\nicefrac{L_i}{T_i})^\delta)}+\gamma\delta L_iT_i^{\delta-1}-1.\]
Because $L_i/T_i^{\delta-1}\to 0$ almost surely as $i\uparrow\infty$ (see Lemma~\ref{lem:small_BC}), we get that, almost surely as $i\uparrow\infty$,
\[|r_i| \leq o(L_i^3 T_i^{\delta -1}),\]
which implies $\sum_{i=1}^n r_i = o(\sum_{i=1}^n L_i^3 T_i^{\delta -1}) = o(n^\delta)$, by Lemma~\ref{lem:sum t_i}, as claimed. This concludes the proof.
\end{proof}

\subsection{Making~\eqref{eq:approx_by_exp} rigorous}
\label{sub:approx_by_exp}

The aim of this section is to prove the following lemma:
\begin{lemma}\label{expectation and variance approximation} 
Almost surely as $n\to \infty$,  \[
\sum_{i=1}^n \frac{F_i W_i}{\bar W_i} 
= \sum_{i=1}^n \frac{\mathbb E_{\bf L}[F_i] W_i}{\bar W_i}  + o(n^{\delta/2}),
\]  
and  
\[\sum_{i=1}^n \frac{F_i^2 W_i}{\bar W_i}\left( 1-\frac{W_i}{\bar W_i}\right) 
=  \sum_{i=1}^n \frac{\mathbb E_{\bf L}[F_i^2] W_i}{\bar W_i} +o(n^{\delta}).
\] 
\end{lemma}

Before proving this lemma, note that Lemmas~\ref{lem:small_expectation}, \ref{lem:small_sec_moment}, and~\ref{expectation and variance approximation} together give that, almost surely as $n\uparrow\infty$,
\begin{equation}\label{eq:small_final_est_sums}
\sum_{i=1}^n \frac{F_i W_i}{\bar W_i}  
= \sigma(n) + o(n^{\nicefrac\delta2})
\quad\text{ and }\quad
\sum_{i=1}^n \frac{F^2_i W_i}{\bar W_i} \left( 1-\frac{W_i}{\bar W_i}\right) 
= \frac{\gamma\mathbb E[L^3]}{3\mathbb E[L]^{1-\delta}} \cdot n^{\delta} +o(n^{\delta}).
\end{equation}

\begin{proof} 
By Lemma~\ref{lem:small_BC}, $L_i/T_{i-1}\to 0$ almost surely as $i\uparrow\infty$. 
Thus, almost surely as $i\uparrow\infty$,
\[T_{i-1}^{\delta}-T_i^{\delta}
= T_{i-1}^{\delta}- (T_{i-1}+L_i)^\delta
= T_{i-1}^\delta \bigg(1-\bigg(1+\frac{L_i}{T_{i-1}}\bigg)^{\!\!\delta}\bigg)
=\mathcal{O}(T_{i-1}^{\delta -1} L_i) = \mathcal O(T_i^{\delta-1}L_i).\]
%Indeed, the second to last equality holds because, for all $i$ large enough, $L_i/T_{i-1}\leq 1$, and the function $x\mapsto \frac1x(1-(1+x)^\delta)$ can be extended continuously to $[0,1]$ and is thus bounded on $[0,1]$.
%The last equality holds because $T_i/T_{i-1}\leq 2$ almost surely for all $i$ large enough.
%(Recall that we say that $a_i = \mathcal O(b_i)$ if there exists a deterministic constant $C>0$ such that, almost surely for all $i$ large enough, $a_i\leq Cb_i$.)
This implies 
\begin{equation} \label{W_i/S_i approximation}
\frac{W_i}{\bar W_i}
=\frac{\mathrm{e}^{\gamma T_{i}^{\delta}}-\mathrm e^{\gamma T_{i-1}^{\delta}}}
{\mathrm{e}^{\gamma T_i^{\delta}}}
=1- \exp\big(-\gamma(T_{i}^{\delta}-T_{i-1}^{\delta})\big)
=\mathcal{O}(T_{i}^{\delta -1} L_i).
\end{equation} 
Thus,  
\[\sum_{i=1}^n \frac{(F_i-\mathbb E_{\bold{L}}[F_i])W_i}{\bar W_i} 
= \mathcal{O}\left( \sum_{i=1}^n (F_i-\mathbb E_{\bold{L}}[F_i])L_iT_{i}^{\delta -1}\right).\]
We apply Theorem~\ref{th:SLLN_Petrov} to $\alpha = 2$, $\Delta_i =  (F_{i}-\mathbb E_{\bf L}[F_{i}])L_{i}T_{i}^{\delta -1}$ and 
$a_i = T_i^{\nicefrac\delta2}$, for all $i\geq 1$. 
Its assumption holds because, using the fact that $F_i\leq L_i$ almost surely for all $i\geq 1$ (see~\eqref{eq:def_F} for the definition of $(F_i)_{i\geq 1}$),
\[\sum_{i\geq1} 
\frac{\mathrm{Var}_{\bold{L}}(F_i)L^2_iT_i^{2\delta-2}}{T_i^\delta} 
\leq  \sum_{i=1}^{\infty} L^4_{i}T_i^{\delta-2}< \infty,\]
by Lemma~\ref{lem:sum t_i} (because, by assumption, $\mathbb EL^4<\infty$).
Thus, Theorem~\ref{th:SLLN_Petrov} applies and gives $\sum_{i=1}^{n} (F_{i}-\mathbb E_{\bf L}[F_{i}])L_{i}T_{i}^{\delta -1} = o(T_n^{\nicefrac\delta2}) = o(n^{\nicefrac\delta2})$ almost surely when $n\uparrow \infty$ (the last equality holds by the strong law of large numbers).
We thus get that, almost surely as $n\uparrow\infty$,
\[\sum_{i=1}^n \frac{F_i W_i}{\bar W_i} 
= \sum_{i=1}^n \frac{\mathbb E_{\bold{L}}[F_i] W_i}{\bar W_i} + o(n^{\nicefrac\delta2}),\]
which concludes the proof of the first statement. 

For the second statement of the lemma, we write
\begin{equation}\label{cond variance error} 
\sum_{i=1}^n \frac{F_i^2 W_i}{\bar W_i}\left( 1-\frac{W_i}{\bar W_i}\right) = \sum_{i=1}^n \frac{F_i^2 W_i}{\bar W_i} - \sum_{i=1}^n \frac{F_i^2 W_i^2}{\bar W_i^2}.
\end{equation} 
For the first sum, we proceed as in the proof of the first statement:
we write
\[\sum_{i=1}^n \frac{F_i^2 W_i}{\bar W_i} 
= \sum_{i=1}^n \frac{\mathbb E_{\bf L}[F_i^2] W_i}{\bar W_i} 
+ \sum_{i=1}^n \frac{(F_i^2-\mathbb E_{\bf L}[F_i^2]) W_i}{\bar W_i},\]
and note that, by~\eqref{W_i/S_i approximation},
\begin{equation}\label{eq:small_truc}
\sum_{i=1}^n \frac{(F_i^2-\mathbb E_{\bf L}[F_i^2]) W_i}{\bar W_i}
= \mathcal O\left(\sum_{i=1}^n \big(F_i^2-\mathbb E_{\bf L}[F_i^2]\big)L_i T_{i}^{\delta-1}\right).
\end{equation}
We now apply Theorem~\ref{th:SLLN_Petrov} to $\alpha=2$, $\Delta_i = \big(F_i^2-\mathbb E_{\bf L}[F_i^2]\big)L_i T_i^{\delta-1}$, and $a_i = T_i^\delta$. Its assumption holds because
\[\sum_{i\geq 1} \frac{\mathrm{Var}_{\bf L}(F_i^2) L_i^2T_i^{2\delta-2}}{T_i^{2\delta}}
\leq \sum_{i\geq 1} L_i^4T_i^{-2}<\infty,
\]
by Lemma~\ref{lem:sum t_i}.
Theorem~\ref{th:SLLN_Petrov} thus implies that, almost surely as $n\uparrow\infty$,
\[\sum_{i=1}^n \big(F_i^2-\mathbb E_{\bf L}[F_i^2]\big)L_i T_i^{\delta-1} = o(n^\delta),\]
and thus, by~\eqref{eq:small_truc},
\begin{equation}\label{eq:small_var1}
\sum_{i=1}^n \frac{(F_i^2-\mathbb E_{\bf L}[F_i^2]) W_i}{\bar W_i} = o(n^\delta).
\end{equation}
For the second sum of~\eqref{cond variance error}, we proceed along the same lines (we skip the technical details) and prove that, almost surely as $n\uparrow\infty$,
\[\sum_{i=1}^n (F_i^2-\mathbb E_{\bold{L}}[F_i^2])\frac{W_i^2}{\bar W_i^2} 
= \mathcal{O}\left( \sum_{i=1}^n (F_i^2-\mathbb E_{\bold{L}}[F_i^2]) L_i^2 T_{i}^{2(\delta -1)}\right)
= o(n^\delta),\]
which implies
\begin{equation}\label{eq:small_truc2}
\sum_{i=1}^n \frac{F_i^2W_i^2}{\bar W_i^2} 
= \sum_{i=1}^n \frac{\mathbb E_{\bf L}[F_i^2]W_i^2}{\bar W_i^2} +o(n^\delta).
\end{equation}
Now,
\[\sum_{i=1}^n \frac{\mathbb E_{\bf L}[F_i^2]W_i^2}{\bar W_i^2} 
=\mathcal O\bigg(\sum_{i=1}^n \mathbb E_{\bf L}[F_i^2]L_i^2T_i^{2(\delta-1)}\bigg)
=\mathcal O\bigg(\sum_{i=1}^n L_i^4 T_i^{2(\delta-1)}\bigg).\]
Applying Theorem~\ref{th:SLLN_Petrov} to $\alpha = 2$, $\Delta_i = L_i^4 T_i^{2(\delta-1)}$, and $a_i = T_i^{\delta}$, we get
\[\sum_{i=1}^n \frac{\mathbb E_{\bf L}[F_i^2]W_i^2}{\bar W_i^2} = o(T_n^\delta) = o(n^\delta).\]
This, together with~\eqref{eq:small_truc2}, gives that, almost surely as $n\uparrow\infty$,
\[\sum_{i=1}^n \frac{F_i^2W_i^2}{\bar W_i^2}  = o(n^\delta).\]
By~\eqref{cond variance error} and~\eqref{eq:small_var1}, this concludes the proof.
\end{proof}

\subsection{Proof of Proposition~\ref{prop:CLT_Phi_small}}\label{sec:Lindeberg}
As mentioned at the beginning of this section, the proof of Proposition~\ref{prop:CLT_Phi_small}
relies on applying Lindeberg's theorem. We apply it, conditionally on ${\bf L}$ and ${\bf F}$, 
to the sum
\[\Phi(n) - \sum_{i=1}^n \frac{F_iW_i}{\bar W_i}
= \sum_{i=1}^n F_i\bigg({\bf 1}_{i\prec n} - \frac{W_i}{\bar W_i}\bigg).\]
Indeed, this is a sum of centred random variables because $\mathbb E_{{\bf L}, {\bf F}}[{\bf 1}_{i\prec n}] = W_i/\bar W_i$, by Lemma~\ref{lem:magic}.
To apply Lindeberg's theorem, we need to show that, for all $\varepsilon>0$,
\[\frac{1}{a_n^2}\sum_{i=1}^n 
\mathbb E_{{\bf L},{\bf F}}[F_i^2({\bf 1}_{i\prec n}-W_i/\bar W_i)^2{\bf 1}_{|F_i({\bf 1}_{i\prec n}-W_i/\bar W_i)|>\varepsilon a_n}] \to 0,\]
as $n\uparrow\infty$, where 
\begin{equation}\label{eq:def_a_n}
a_n^2:= \sum_{i=1}^n \mathrm{Var}_{{\bf L}, {\bf F}}\big(F_i{\bf 1}_{i\prec n}\big)
= \sum_{i=1}^n F_i^2 \mathrm{Var}_{{\bf L}, {\bf F}}\big({\bf 1}_{i\prec n}\big)
= \sum_{i=1}^n \frac{F_i^2W_i}{\bar W_i}\bigg(1-\frac{W_i}{\bar W_i}\bigg),
\end{equation}
by Lemma~\ref{lem:magic}.
First note that, by definition of $(F_i)_{i\geq 1}$ (see~\eqref{eq:def_F}), $F_i\leq L_i$ almost surely for all $i\geq 1$. Also, ${\bf 1}_{i\prec n}-W_i/\bar W_i\in [-1, 1]$ almost surely for all $i\geq 1$. Thus, almost surely, 
\ba
\frac{1}{a_n^2}\sum_{i=1}^n 
\mathbb E_{{\bf L},{\bf F}}[F_i^2({\bf 1}_{i\prec n}-W_i/\bar W_i)^2{\bf 1}_{|F_i({\bf 1}_{i\prec n}-W_i/\bar W_i)|>\varepsilon a_n}] 
&\leq \frac1{a_n^2}\sum_{i=1}^n L_i^2 \mathbb P_{{\bf L},{\bf F}}\big(|F_i({\bf 1}_{i\prec n}-W_i/\bar W_i)|>\varepsilon a_n\big)\\
&\leq \frac1{\varepsilon^2 a_n^4}\sum_{i=1}^n L_i^2 \mathbb E_{{\bf L},{\bf F}}\big[|F_i({\bf 1}_{i\prec n}-W_i/\bar W_i)|^2\big],
\ea
by Markov inequality. Using again the fact that $|F_i({\bf 1}_{i\prec n}-W_i/\bar W_i)|\leq L_i$ almost surely for all $i\geq 1$, we get 
\ba
\frac{1}{a_n^2}\sum_{i=1}^n 
\mathbb E_{{\bf L},{\bf F}}[F_i^2({\bf 1}_{i\prec n}-W_i/\bar W_i)^2{\bf 1}_{|F_i({\bf 1}_{i\prec n}-W_i/\bar W_i)|>\varepsilon a_n}] 
&\leq \frac1{\varepsilon^2 a_n^4}\sum_{i=1}^n L_i^4 
= \frac{n\mathbb E[L^4]}{\varepsilon^2 a_n^4},
\ea
almost surely as $n\uparrow\infty$, by the strong law of large numbers.
Now, by~\eqref{eq:small_final_est_sums}, almost surely as $n\uparrow\infty$,
\[a_n^2 = \frac{\gamma\mathbb E[L^3]}{3\mathbb E[L]^{1-\delta}} \cdot n^{\delta} +o(n^{\delta}).\]
Thus,
\[\frac{1}{a_n^2}\sum_{i=1}^n 
\mathbb E_{{\bf L},{\bf F}}[F_i^2({\bf 1}_{i\prec n}-W_i/\bar W_i)^2{\bf 1}_{|F_i({\bf 1}_{i\prec n}-W_i/\bar W_i)|>\varepsilon a_n}] = \mathcal O(n^{1-2\delta}) \to 0,\]
almost surely as $n\uparrow\infty$ because $\delta>\nicefrac12$.
Thus, Lindeberg's theorem applies and gives that, conditionally on ${\bf L}$ and ${\bf F}$, in distribution as $n\uparrow\infty$,
\[\frac{\Phi(n) - \sum_{i=1}^n \frac{F_iW_i}{\bar W_i}}{a_n}
\Rightarrow\mathcal N(0,1).\]
By~\eqref{eq:def_a_n} and~\eqref{eq:small_final_est_sums}, this implies
\[\frac{\Phi(n) - \sigma(n)+o(n^{\nicefrac\delta2})}
{\sqrt{\frac{\gamma\mathbb E[L^3]}{3\mathbb E[L]^{1-\delta}} \cdot n^{\delta}}}
\Rightarrow\mathcal N(0,1),\]
which concludes the proof of Proposition~\ref{prop:CLT_Phi_small}.

\section{Proof of Theorem~\ref{th:CLT_crit}}\label{sec:crit}
In this section, we assume that $\delta = 1$.
The proof follows the same route as in the case $\delta <1$, and the key ingredient is the following limiting theorem for $\Phi(n)$:
\begin{proposition}\label{CLT Phi delta=1} 
Under the assumptions of Theorem~\ref{th:CLT_crit} and if $\delta = 1$, then, in distribution as $n\uparrow\infty$,
\[\frac{\Phi(n) - n\mathbb E[L-(1-\mathrm e^{-\gamma L})/\gamma]}{\sqrt n}\Rightarrow \Omega_1 + \Omega_2,\]
where
\begin{equation}\label{eq:def_Xi}
\Omega_1\sim \mathcal N\bigg(0, \frac1{\gamma^2} + \mathbb E\bigg[L^2-\bigg(L+\frac{\mathrm e^{-\gamma L}}{\gamma}\bigg)^{\!\!2}\bigg]\bigg)\quad \text{ and }\quad
\Omega_2\sim \mathcal N\bigg(0,\mathrm{Var}\bigg(L- \frac{1-\mathrm e^{-\gamma L}}\gamma\bigg)\bigg),
\end{equation}
and
\[\mathrm{Cov}(\Omega_1, \Omega_2) 
= \mathrm{Var}\bigg(L - \frac{1- \mathrm e^{-\gamma L}}\gamma\bigg).\] 
\end{proposition}

Before proving Proposition~\ref{CLT Phi delta=1}, we show how it implies Theorem \ref{th:CLT_crit}:

\begin{proof}[Proof of Theorem~\ref{th:CLT_crit}] 
By the strong law of large numbers, $i(t)\to \infty$ almost surely as $t\uparrow\infty$. Thus, by Proposition~\ref{CLT Phi delta=1}, in distribution as $t\uparrow\infty$,
\[\frac{\Phi(i(t)-1)-i(t)\mathbb E[1-(1-\mathrm e^{-\gamma L})/\gamma]}{\sqrt{i(t)}} \Rightarrow \Omega_1+\Omega_2.\]
Now, by the central limit theorem, in distribution as $n\uparrow\infty$
\begin{equation}\label{eq:def_Phi}
\frac{\sum_{i=1}^n L_i - n\mathbb EL}{\sqrt n} \Rightarrow \Psi \sim \mathcal N(0, \mathrm{Var}(L)).
\end{equation}
Thus, because $i(t)\uparrow\infty$ almost surely as $t\uparrow\infty$, we get
\[\frac{\sum_{i=1}^{i(t)} L_i - i(t)\mathbb EL}{\sqrt{i(t)}} \Rightarrow \Psi.\]
By standard results in renewal theory, $L_{i(t)}$ converges in distribution as $t\uparrow\infty$. Because $T_{i(t)}-L_{i(t)}\leq t<T_{i(t)}$, this implies that $T_{i(t)} = \sum_{i=1}^{i(t)} L_i = t+\mathcal O(1)$ in probability as $t\uparrow\infty$.
Thus, in distribution as $t\uparrow\infty$,
\[\frac{t - i(t)\mathbb EL}{\sqrt{i(t)}} \Rightarrow \Psi.\]
This implies
\[\frac{\Phi(i(t)-1)- t\mathbb E[L-(1-\mathrm e^{-\gamma L})/\gamma]/\mathbb E[L]}{\sqrt{i(t)}}
\Rightarrow \Omega_1+\Omega_2 + \Omega_3,\]
where we have set $\Omega_3 = \Psi{\mathbb E[L-(1-\mathrm e^{-\gamma L})/\gamma]}/{\mathbb E[L]}$.
Because $A(t)/\sqrt{i(t)}\to 0$ in probability, by Lemma~\ref{lem:renewal}, and because, by definition, $S(t) = \Phi(i(t)-1)+A(t)$, this implies that
\[\frac{S(t)- t\mathbb E[L-(1-\mathrm e^{-\gamma L})/\gamma]/\mathbb E[L]}{\sqrt{i(t)}}
\Rightarrow \Omega_1+\Omega_2 + \Omega_3,\]
which implies
\[\frac{S(t)- t\mathbb E[L-(1-\mathrm e^{-\gamma L})/\gamma]/\mathbb E[L]}{\sqrt{t}}
\Rightarrow \frac1{\sqrt{\mathbb EL}}(\Omega_1+\Omega_2 + \Omega_3) =: \Omega
\]
By Skorokhod's representation theorem, there exist a probability space (which we call Skorokhod's probability space) on which one can define a process $(\tilde{S}(t))_{t\geq 0}$ and a random variable $\tilde \Omega$ such that $\tilde S(t) = S(t)$ in distribution for all $t\geq 0$, $\tilde \Omega = \Omega$ in distribution, and
\[\frac{\tilde S(t)- t\mathbb E[L-(1-\mathrm e^{-\gamma L})/\gamma]/\mathbb E[L]}{\sqrt{i(t)}}
\to \tilde\Omega,\]
almost surely as $t\uparrow\infty$.
In Skorokhod's probability space, there also exist a process $(\tilde{Z}(t))_{t\geq 0}$ and a random variable $\tilde\Gamma$ of distribution $\gamma$ such that $(\tilde{Z}(t))_{t\geq 0}$ is independent of~$(\tilde{S}(t))_{t\geq 0}$, $\tilde Z(t) = Z(t)$ in distribution for all $t\geq 0$, and
\[\frac{\tilde{Z}(t)-a(t)}{b(t)}\to \tilde\Gamma,\]
almost surely as $t\uparrow\infty$. 
Thus, by Assumption {\bf (A'2)},
\begin{align*}
\frac{\tilde{Z}(\tilde{S}(t))-a(s(t))}{b(s(t))} 
&= \frac{\tilde{Z}(\tilde{S}(t))-a(\tilde{S}(t))}{b(\tilde{S}(t))} \cdot \frac{b(\tilde{S}(t))}{a(s(t))}  + \frac{a(\tilde{S}(t))-a(s(t))}{b(s(t))}\\
& \to f(\tilde\Omega)+\tilde\Gamma g(\tilde\Omega).
\end{align*}
This implies convergence in distribution on the original probability space. Thus, it only remains to calculate $\mathrm{Cov}(\Omega_1, \Omega_3)$ and $\mathrm{Cov}(\Omega_2, \Omega_3)$; this is done in Section~\ref{sub:covariances}, where we prove that
\begin{equation}\label{eq:covariances}
\mathrm{Cov}(\Omega_1, \Psi) 
=\mathrm{Cov}(\Omega_2, \Psi) 
=\mathrm{Cov}(L, L+ \mathrm e^{-\gamma L}/\gamma).
\end{equation}
This concludes the proof of Theorem~\ref{th:CLT_crit} because $\Omega_3 = \Psi{\mathbb E[L-(1-\mathrm e^{-\gamma L})/\gamma]}/{\mathbb E[L]}$.
\end{proof}

As in Section~\ref{sec:small}, the idea to prove Proposition~\ref{CLT Phi delta=1} is to
apply Lindeberg's theorem. However, this time, we only condition on {\bf L} (and not on {\bf F}).
This will give
\begin{equation}\label{eq:crit_Lind}
\frac{\Phi(n)-\mathbb E_{{\bf L}}[\Phi(n)]}
{\sqrt{\mathrm{Var}_{{\bf L}}(\Phi(n))}}\Rightarrow \mathcal N(0,1).
\end{equation}
By Lemma~\ref{lem:magic}, and because, 
given ${\bf L}$, the sequences $(F_i)_{1\leq i\geq n}$ and $({\bf 1}_{i\prec n})_{1\leq i\leq n}$ are independent, we get
\begin{equation}\label{eq:crit_exp}
\mathbb E_{\bf L}[\Phi(n)]
=\sum_{i=1}^{n} \frac{W_i\mathbb E_{\bf L}[F_i]}{\bar W_i}.
\end{equation}
For the variance, we use the fact that, given ${\bf L}$, $(F_i{\bf 1}_{i\prec n})_{1\leq i\leq n}$ is a sequence of independent random variables, and thus 
\[\mathrm{Var}_{\bf L}(\Phi(n)) = \sum_{i=1}^n \mathrm{Var}_{\bf L}(F_i {\bf 1}_{i\prec n}).\]
Using the total variance law in the first equality and Lemma~\ref{lem:magic} in the second one, we get that, for all $1\leq i\leq n$,
\ban
\mathrm{Var}_{\bf L}(F_i {\bf 1}_{i\prec n})
&= \mathrm{Var}_{\bf L}\big(\mathbb E_{{\bf L},{\bf F}}[F_i {\bf 1}_{i\prec n}]\big)
+ \mathbb E_{\bf L}\big[\mathrm{Var}_{{\bf L},{\bf F}}(F_i {\bf 1}_{i\prec n})\big]
= \mathrm{Var}_{\bf L}\bigg(\frac{F_iW_i}{\bar W_i}\bigg)
+ \mathbb E_{\bf L}\bigg[\frac{F_i^2W_i}{\bar W_i}\bigg(1-\frac{W_i}{\bar W_i}\bigg)\bigg]\notag\\
&= \bigg(\frac{W_i}{\bar W_i}\bigg)^2 \mathrm{Var}_{\bf L}(F_i) + \frac{W_i}{\bar W_i}\bigg(1-\frac{W_i}{\bar W_i}\bigg)\mathbb E_{\bf L}[F_i^2]
= \frac{W_i}{\bar W_i} \mathbb E_{\bf L}[F_i^2] - \bigg(\frac{W_i}{\bar W_i}\bigg)^{\!\!2}\mathbb E_{\bf L}[F_i]^2.
\label{eq:total_variance}
\ean
We thus get
\begin{equation}\label{eq:crit_var}
\mathrm{Var}_{\bf L}(\Phi(n)) 
= \sum_{i=1}^n\bigg(\frac{W_i}{\bar W_i} \mathbb E_{\bf L}[F_i^2] - \bigg(\frac{W_i}{\bar W_i}\mathbb E_{\bf L}[F_i]\bigg)^{\!\!2} \bigg).
\end{equation}
The rest of the section is dedicated to proving Proposition~\ref{CLT Phi delta=1}: In Section~\ref{sec:crit_estimates}, we give asymptotic estimates for $\mathbb E_{\bf L}[\Phi(n)]$ and 
$\mathrm{Var}_{\bf L}(\Phi(n))$. In Section~\ref{sec:crit_lindeberg}, we apply Lindeberg theorem to prove~\eqref{eq:crit_Lind} and conclude the proof.

\subsection{Estimating the expectation and variance of $\Phi(n)$}\label{sec:crit_estimates}
In this section, we prove the following lemma:
\begin{lemma}\label{lem:crit_exp_and_var}
Almost surely for all $n\geq 1$,
\[\mathbb E_{\bf L}[\Phi(n)] = 
\sum_{i=1}^n\bigg(L_i - \frac{1- \mathrm e^{-\gamma L_i}}\gamma\bigg),\]
and
\[\mathrm{Var}_{\bf L}(\Phi(n)) = \frac n{\gamma^2}+
\sum_{i=1}^n \bigg(L^2_i-\bigg(L_i+\frac{\mathrm e^{-\gamma L_i}}{\gamma}\bigg)^{\!\!2}\bigg).\]
\end{lemma}

Before proving this result, we note that, by the central limit theorem for $\mathbb E_{\bf L}[\Phi(n)]$
and the law of large numbers for $\mathrm{Var}_{\bf L}(\Phi(n))$,
\begin{equation}\label{eq:crit_est_exp}
\frac1{\sqrt n}\bigg(\mathbb E_{\bf L}[\Phi(n)] - n\mathbb E\bigg[L - \frac{1- \mathrm e^{-\gamma L}}\gamma\bigg]\bigg)
\Rightarrow \mathcal N\bigg(0,\mathrm{Var}\bigg(L - \frac{1- \mathrm e^{-\gamma L}}\gamma\bigg) \bigg),
\end{equation}
in distribution as $n\uparrow\infty$, and
\begin{equation}\label{eq:crit_est_var}
\mathrm{Var}_{\bf L}(\Phi(n)) =\frac n{\gamma^2}+n\mathbb E\bigg[L^2-\bigg(L+\frac{\mathrm e^{-\gamma L}}{\gamma}\bigg)^{\!\!2}\bigg] + o(n),
\end{equation}
almost surely as $n\uparrow\infty$.

\begin{proof}
First note that, in the case $\delta = 1$, for all $i\geq 1$,
\begin{equation}\label{eq:W/S}
\frac{W_i}{\bar W_i} = \frac{\bar W_i - \bar W_{i-1}}{\bar W_i} 
= \frac{\mathrm e^{\gamma T_i} - \mathrm e^{\gamma T_{i-1}}}{\mathrm e^{\gamma T_i}}
= 1- \mathrm e^{-\gamma L_i}.
\end{equation}
Also, using the definition of $(F_i)_{i\geq 1}$ (see~\eqref{eq:def_F}) in the first equality 
and integration by parts in the third one, we get that, for all $i\geq 1$,
\ban
W_i \mathbb E_{\bf L}[F_i]
&= \int_0^{L_i} \gamma x\mathrm e^{\gamma(T_{i-1}+x)}\mathrm dx
= \int_{T_{i-1}}^{T_i} \gamma (x-T_{i-1})\mathrm e^{\gamma x}\mathrm dx
= \big[(x-T_{i-1})\mathrm e^{\gamma x}\big]_{T_{i-1}}^{T_i}
- \int_{T_{i-1}}^{T_i} \mathrm e^{\gamma x}\mathrm dx\label{eq:crit_truc}\\
&= L_i \mathrm e^{\gamma T_i} - \frac{\mathrm e^{\gamma T_i}-\mathrm e^{\gamma T_{i-1}}}\gamma,\notag
\ean
because, $\mu(x) = \gamma \mathrm e^{\gamma x}$ in the case when $\delta =1$.
Using~\eqref{eq:W/S}, we thus get that, for all $i\geq 1$,
\begin{equation}\label{eq:crit_WEF/S}
\frac{W_i \mathbb E_{\bf L}[F_i]}{\bar W_i}= L_i - \frac{1- \mathrm e^{-\gamma L_i}}\gamma.
\end{equation}
Thus, by~\eqref{eq:crit_exp}, we get
\[\mathbb E_{\bf L}[\Phi(n)] 
= \sum_{i=1}^n\bigg(L_i - \frac{1- \mathrm e^{-\gamma L_i}}\gamma\bigg).\]
For the variance, recall that, by~\eqref{eq:crit_var},
\begin{equation}\label{eq:crit_var2}
\mathrm{Var}_{\bf L}(\Phi(n)) 
= \sum_{i=1}^n\bigg(\frac{W_i}{\bar W_i} \mathbb E_{\bf L}[F_i^2] - \bigg(\frac{W_i}{\bar W_i}\mathbb E_{\bf L}[F_i]\bigg)^{\!\!2}\bigg).
\end{equation}
By~\eqref{eq:crit_WEF/S}, we have
\begin{equation}\label{eq:crit_var_first_term}
\bigg(\frac{W_i}{\bar W_i}\mathbb E_{\bf L}[F_i]\bigg)^2
=\bigg(L_i - \frac{1- \mathrm e^{-\gamma L_i}}\gamma\bigg)^{\!\!2}.
\end{equation}
Also, by definition of $(F_i)_{i\geq 1}$, for all $i\geq 1$,
\ba
W_i\mathbb E_{\bf L}[F_i^2]
&= \int_0^{L_i} \gamma x^2 \mathrm e^{\gamma(T_{i-1}+x)}\mathrm dx
= \int_{T_{i-1}}^{T_i} \gamma (x-T_{i-1})^2 \mathrm e^{\gamma x}\mathrm dx\\
&= \big[(x-T_{i-1})^2 \mathrm e^{\gamma x}\big]_{T_{i-1}}^{T_i}
- 2\int_{T_{i-1}}^{T_i}  (x-T_{i-1}) \mathrm e^{\gamma x}\mathrm dx
= L_i^2 \bar W_i - \frac2\gamma W_i\mathbb E_{\bf L}[F_i],
\ea
where we have used the fact that $W_i\mathbb E_{\bf L}[F_i] = \int_{T_{i-1}}^{T_i}  \gamma(x-T_{i-1}) \mathrm e^{\gamma x}\mathrm dx$ (as seen in~\eqref{eq:crit_truc}).
Now using~\eqref{eq:crit_WEF/S}, we get
\[\frac{W_i\mathbb E_{\bf L}[F_i^2]}{\bar W_i} 
= L_i^2 - \frac2{\gamma}\bigg(L_i - \frac{1- \mathrm e^{-\gamma L_i}}\gamma\bigg).\]
This, together with~\eqref{eq:crit_var_first_term}, gives
\[\frac{W_i}{\bar W_i} \mathbb E_{\bf L}[F_i^2] - \bigg(\frac{W_i}{\bar W_i}\mathbb E_{\bf L}[F_i]\bigg)^2
=L_i^2 - \frac2{\gamma}\bigg(L_i - \frac{1- \mathrm e^{-\gamma L_i}}\gamma\bigg)-\bigg(L_i - \frac{1- \mathrm e^{-\gamma L_i}}\gamma\bigg)^{\!\!2}
= L^2_i+ \frac1{\gamma^2}-\bigg(L_i+\frac{\mathrm e^{-\gamma L_i}}{\gamma}\bigg)^{\!\!2},\]
which concludes the proof.
\end{proof}

\subsection{Applying Lindeberg's theorem}\label{sec:crit_lindeberg}
In this section, we prove that, conditionally on ${\bf L}$, in distribution as $n\uparrow\infty$,
\begin{equation}\label{eq:crit_Lind2}
\frac{\Phi(n)-\mathbb E_{{\bf L}}[\Phi(n)]}
{\sqrt{\mathrm{Var}_{{\bf L}}(\Phi(n))}}\Rightarrow \mathcal N(0,1).
\end{equation}
To do so, we apply Lindeberg's theorem to the triangular array 
$(F_i {\bf 1}_{i\prec n})_{1\leq i\leq n}$, conditionally on ${\bf L}$.
For all $n\geq 1$, we let
\[a_n^2 
= \sum_{i=1}^n \mathrm{Var}_{\bf L}\big(F_i {\bf 1}_{i\prec n}\big)
= \mathrm{Var}_{{\bf L}}(\Phi(n)) 
\sim \frac n{\gamma^2}+n\mathbb E\bigg[L-\bigg(1+\frac{\mathrm e^{-\gamma L}}{\gamma}\bigg)^{\!\!2}\bigg],\]
almost surely as $n\uparrow\infty$.
To apply Lindeberg's theorem, we need to show that
\begin{equation}\label{eq:a_n_to0}
\frac1{a_n^2}\sum_{i=1}^n \mathbb E_{\bf L}\big[
(F_i {\bf 1}_{i\prec n}-\mathbb E_{\bf L}[F_i {\bf 1}_{i\prec n}])^2
{\bf 1}_{|F_i {\bf 1}_{i\prec n}-\mathbb E_{\bf L}[F_i {\bf 1}_{i\prec n}|>\varepsilon a_n}
\big]
\to 0,
\end{equation}
almost surely as $n\uparrow\infty$, for all $\varepsilon>0$.
Using the fact that, almost surely, $0\leq F_i{\bf 1}_{i\prec n}\leq L_i$ for all $1\leq i\leq n$, we get
\ba
\frac1{a_n^2}\sum_{i=1}^n \mathbb E_{\bf L}\big[
(F_i {\bf 1}_{i\prec n}-\mathbb E_{\bf L}[F_i {\bf 1}_{i\prec n}])^2
{\bf 1}_{|F_i {\bf 1}_{i\prec n}-\mathbb E_{\bf L}[F_i {\bf 1}_{i\prec n}|>\varepsilon a_n}
\big]
&\leq \frac1{a_n^2}\sum_{i=1}^n L_i^2\mathbb P_{\bf L}\big(
|F_i {\bf 1}_{i\prec n}-\mathbb E_{\bf L}[F_i {\bf 1}_{i\prec n}|>\varepsilon a_n
\big)\\
&\leq \frac1{a_n^2}\sum_{i=1}^n L_i^2\cdot \frac{\mathrm{Var}_{\bf L}(F_i {\bf 1}_{i\prec n})}{\varepsilon^2 a_n^2},
\ea
by Markov's inequality. Now, using again the fact that $0\leq F_i{\bf 1}_{i\prec n}\leq L_i$ almost surely for all $1\leq i\leq n$, we get
\[\frac1{a_n^2}\sum_{i=1}^n \mathbb E_{\bf L}\big[
(F_i {\bf 1}_{i\prec n}-\mathbb E_{\bf L}[F_i {\bf 1}_{i\prec n}])^2
{\bf 1}_{|F_i {\bf 1}_{i\prec n}-\mathbb E_{\bf L}[F_i {\bf 1}_{i\prec n}|>\varepsilon a_n}
\big]
\leq \frac1{\varepsilon^2 a_n^4}\sum_{i=1}^n L_i^4 \sim \frac{\mathbb E[L^4]n}{\varepsilon^2 a_n^4}\to 0,
\]
almost surely as $n\uparrow\infty$, by~\eqref{eq:a_n_to0}.
This concludes the proof of~\eqref{eq:crit_Lind2}.

\subsection{Proof of Proposition \ref{CLT Phi delta=1}}\label{sec:crit_proof}
The proof of Proposition~\ref{CLT Phi delta=1} relies on Equations~\eqref{eq:crit_est_exp}, \eqref{eq:crit_est_var}~ and~\eqref{eq:crit_Lind2}.
Indeed, we write
\ba
\frac{\Phi(n) - n\mathbb E[1-(1-\mathrm e^{\gamma L})/\gamma]}{\sqrt n}
&= \frac{\Phi(n) - \mathbb E_{\bf L}[\Phi(n)]}{\sqrt{\mathrm{Var}_{\bf L}(\Phi(n))}}
\cdot \sqrt{\frac{\mathrm{Var}_{\bf L}(\Phi(n))}{n}}
+ \frac{\mathbb E_{\bf L}[\Phi(n)] - n\mathbb E[1-(1-\mathrm e^{\gamma L})/\gamma]}{\sqrt n}\\
&\Rightarrow \Omega_1 + \Omega_2,
\ea
where $\Omega_1$ and $\Omega_2$ are as in~\eqref{eq:def_Xi}.
To prove Proposition \ref{CLT Phi delta=1}, 
it only remains to calculate $\mathrm{Cov}(\Omega_1, \Omega_2)$.
To do so, note that
\[\mathrm{Cov}\big(\Phi(n), \mathbb E_{\bf L}[\Phi(n)]\big)
=\mathbb E\big[\Phi(n) \mathbb E_{\bf L}[\Phi(n)]\big]
-\mathbb E[\Phi(n)] \mathbb E[\mathbb E_{\bf L}[\Phi(n)]]
= \mathbb E\big[\mathbb E_{\bf L}[\Phi(n)]^2\big]
-\mathbb E\big[\mathbb E_{\bf L}[\Phi(n)]\big]^2,
\]
by the tower rule.
Thus,
\[\mathrm{Cov}\big(\Phi(n), \mathbb E_{\bf L}[\Phi(n)]\big)
=\mathrm{Var}(\mathbb E_{\bf L}[\Phi(n)])
= \mathrm{Var}\bigg(\sum_{i=1}^n\bigg(L_i - \frac{1- \mathrm e^{-\gamma L_i}}\gamma\bigg)\bigg),\]
by Lemma~\ref{lem:crit_exp_and_var}.
By independence,
\[\mathrm{Cov}\big(\Phi(n), \mathbb E_{\bf L}[\Phi(n)]\big)
= n\mathrm{Var}\bigg(L - \frac{1- \mathrm e^{-\gamma L}}\gamma\bigg).\]
Thus,
\ba\mathrm{Cov}(\Omega_1,\Omega_2)
&=\lim_{n\uparrow\infty}\mathrm{Cov}\bigg(\frac{\Phi(n) - \mathbb E_{\bf L}[\Phi(n)]}{\sqrt{n}}, 
\frac{\mathbb E_{\bf L}[\Phi(n)] - n\mathbb E[1-(1-\mathrm e^{\gamma L})/\gamma]}{\sqrt n}\bigg)\\
&= \lim_{n\uparrow\infty} \frac1n \mathrm{Cov}\big(\Phi(n), \mathbb E_{\bf L}[\Phi(n)]\big)
= \mathrm{Var}\bigg(L - \frac{1- \mathrm e^{-\gamma L}}\gamma\bigg),
\ea
which concludes the proof.

\subsection{Covariances}\label{sub:covariances}
The aim of this section is to prove Equation~\eqref{eq:covariances}.
Note that, for all $n\geq 1$,
\ba\mathrm{Cov}\bigg(\sum_{i=1}^n L_i, \mathbb E_{\bf L}[\Phi(n)]\bigg)
&= \mathrm{Cov}\bigg(\sum_{i=1}^n L_i, \sum_{i=1}^n\bigg(1 - \frac{1- \mathrm e^{-\gamma L_i}}\gamma\bigg)\bigg)
= \sum_{i=1}^n \sum_{j=1}^n  \mathrm{Cov}\bigg(L_i, \bigg(1 - \frac{1- \mathrm e^{-\gamma L_j}}\gamma\bigg)\bigg)\\
&= \sum_{i=1}^n  \mathrm{Cov}\bigg(L_i, \bigg(L_i - \frac{1- \mathrm e^{-\gamma L_i}}\gamma\bigg)\bigg)
= n\mathrm{Cov}\bigg(L, \bigg(L - \frac{1- \mathrm e^{-\gamma L}}\gamma\bigg)\bigg)\\
&= n\mathrm{Var}(L) + n\mathrm{Cov}(L, \mathrm e^{-\gamma L}/\gamma).
\ea
Thus, by definition of $\Psi$ (see~\eqref{eq:def_Phi}) and $\Omega_2$ (see~\eqref{eq:def_Xi})
\[\mathrm{Cov}(\Psi, \Omega_2)
= \lim_{n\to\infty} \mathrm{Cov}\bigg(\frac{\sum_{i=1}^n L_i - n\mathbb EL}{\sqrt n},
\frac{\mathbb E_{\bf L}[\Phi(n)] - n\mathbb E[1-(1-\mathrm e^{\gamma L})/\gamma]}{\sqrt n}\bigg)
= \mathrm{Var}(L) + \mathrm{Cov}(L, \mathrm e^{-\gamma L}/\gamma).\]
Finally, by the tower rule,
\[\mathrm{Cov}\bigg(\sum_{i=1}^n L_i, \Phi(n)\bigg)
= \mathrm{Cov}\bigg(\sum_{i=1}^n L_i, \mathbb E_{\bf L}[\Phi(n)]\bigg),
\]
which concludes the proof of~\eqref{eq:covariances}.

\section{Proof of Theorem~\ref{th:CLT_large}}\label{sec:large}
In this section, we assume that $\delta >1$.
As in the two previous cases, we need to first understand the asymptotic behaviour of $\Phi(n)$:
\begin{proposition}\label{Phi CLT delta > 1} 
If $\delta>1$, under the assumptions of Theorem~\ref{th:CLT_large}, then, conditionally on ${\bf L}$,
%If $\mathbb E[L^4]<\infty$, there exists a random variable $\Psi$ such that conditionally on $(\bold{L},\bold{F})$, $(\bold{L},\bold{F})$-almost surely, we have 
almost surely as $n\uparrow\infty$,
\begin{equation}\label{eq:large_CLT_Phi}
\Phi(n) =  \begin{cases}
T_n + \mathcal O(1) & \text{ if }\delta >2\\[5pt]
\displaystyle T_n - \sum_{i=1}^n \frac1{\gamma\delta T_i^{\delta-1}} + \mathcal O(1)
& \text{ if }\delta \in (\nicefrac32, 2]\\
\displaystyle T_n - \sum_{i=1}^n \frac1{\gamma\delta T_i^{\delta-1}} + o\big(n^{\nicefrac32-\delta}(\log n)^2\big)
& \text{ if }\delta \in (1, \nicefrac32],
\end{cases}
\end{equation}
\end{proposition}
Before proving Proposition~\ref{Phi CLT delta > 1}, we how how it implies Theorem~\ref{th:CLT_large}:

\begin{proof}[Proof of Theorem \ref{th:CLT_large}]
Since, by the strong law of large numbers, $i(t)\uparrow\infty$ almost surely as $t\uparrow\infty$, 
Proposition~\ref{Phi CLT delta > 1} implies that, almost surely as $t\uparrow\infty$,
\begin{equation}\label{CLT for Phi(u_i(t))}
\Phi(i(t)-1)
=   \begin{cases}
T_{i(t)-1} + \mathcal O(1) & \text{ if }\delta >2\\[5pt]
\displaystyle T_{i(t)-1} - \sum_{i=1}^{i(t)-1} \frac1{\gamma\delta T_i^{\delta-1}} + \mathcal O(1)
& \text{ if }\delta \in (\nicefrac32, 2]\\[5pt]
\displaystyle T_{i(t)-1} - \sum_{i=1}^{i(t)-1} \frac1{\gamma\delta T_i^{\delta-1}} + o(\sqrt{i(t)})
& \text{ if }\delta \in (1, \nicefrac32]
\end{cases}
\end{equation}
Recall that, by definition, $i(t)$ is the integer such that $T_{i(t)-1}\leq t < T_{i(t)-1} + L_{i(t)} = T_{i(t)}$.
By standard renewal theory, $L_{i(t)}$ converges in distribution as $t\uparrow\infty$ (see the proof of Lemma~\ref{lem:renewal}), and thus, $L_{i(t)} = \mathcal O(1)$ in probability as $t\uparrow\infty$. %implying that, in probability as $t\uparrow\infty$, $T_{i(t)-1} = t + \mathcal O(1)$.
%This implies that, for $\delta>2$,
%\[\Phi(i(t)-1) = t + \mathcal O(1),\]
%in probability as $t\uparrow\infty$.
For $\delta\in (1,2]$, we have, on the one hand,
\[\sum_{i=1}^{i(t)-1}\frac1{\gamma\delta T_i^{\delta-1}}
\leq \sum_{i=1}^{i(t)-1}\int_{T_{i-1}}^{T_i} \frac1{\gamma\delta x^{\delta-1}} \mathrm dx
= \int_{0}^{T_{i(t)-1}} \frac1{\gamma\delta x^{\delta-1}} \mathrm dx
= \frac{T_{i(t)-1}^{2-\delta}}{\gamma\delta (2-\delta)} 
= \frac{t^{2-\delta}}{\gamma\delta (2-\delta)} + \mathcal O(1),
\]
and, on the other hand,
\[\sum_{i=1}^{i(t)-1}\frac1{\gamma\delta T_i^{\delta-1}}
\geq \sum_{i=1}^{i(t)-1}\int_{T_i}^{T_{i+1}} \frac1{\gamma\delta x^{\delta-1}} \mathrm dx
= \int_{T_1}^{T_{i(t)}} \frac1{\gamma\delta x^{\delta-1}} \mathrm dx
= \frac{T_{i(t)}^{2-\delta} - T_1^{2-\delta}}{\gamma\delta (2-\delta)}.
\]
Because $T_{i(t)} = t + \mathcal O(1)$ in probability as $t\uparrow\infty$, and because $T_1 = L_1 = \mathcal O(1)$ as well, we get
\[\sum_{i=1}^{i(t)-1}\frac1{\gamma\delta T_i^{\delta-1}} 
= \frac{t^{2-\delta}}{\gamma\delta (2-\delta)} + \mathcal O(1),
\]
in probability as $t\uparrow\infty$, if $\delta\in(1, 2]$.
Now recall that, in distribution for all $t\geq 0$, $S(t) = \Phi(i(t)-1) + A(t)$ (see~\eqref{eq:def_S}).
%By Lemma~\ref{lem:renewal}, $A(t) = \mathcal O(1)$ in probability as $t\uparrow\infty$.
Thus, because $T_{i(t)-1} + A(t) = t$, by definition, we get
\[S(t) = \Phi(i(t)) + A(t) = 
\begin{cases}
t + \mathcal O(1) & \text{ if }\delta >2\\[3pt]
\displaystyle t - \frac{t^{2-\delta}}{\gamma\delta(2-\delta)\mathbb E[L]} + \mathcal O(1)
& \text{ if }\delta \in (\nicefrac32, 2]\\[7pt]
\displaystyle t - \frac{t^{2-\delta}}{\gamma\delta(2-\delta)\mathbb E[L]} + o(t^{\nicefrac32-\delta}(\log t)^2) & \text{ if }\delta \in (1, \nicefrac32].
\end{cases}\]
By definition of $s(t)$ (see~\eqref{eq:def_s} for $\delta>1$), 
we get 
\ba
\frac{{Z}({S}(t))-a(s(t))}{b(s(t))}
& = \frac{{Z}({S}(t))-a(S(t))}{b(S(t))}\cdot \frac{b({S}(t))}{b(s(t))} + \frac{a({S}(t))-a(s(t))}{b(s(t))}\\
& \Rightarrow \Gamma,
\ea
by Assumption {\bf (A2$_{\bs\delta}$)}.
This concludes the proof because $X(t) = Z(S(t))$ in distribution for all $t\geq 0$ (see Section~\ref{sec:prel}).
\end{proof}

\subsection{Preliminary lemmas}
We start with the following preliminary lemmas:
\begin{lemma}\label{lem:large_BC}
For all $\delta>1$, if $\mathbb E[L^2]<\infty$,
then, for all $\varepsilon>0$, 
almost surely for all $i$ large enough, $L_i/T_i\leq \varepsilon$.
\end{lemma}

\begin{proof}
We prove this using Borel-Cantelli's lemma.
Indeed, first note that, by the strong law of large numbers, almost surely for all $i$ large enough, $T_i\geq i\mathbb E[L]/2$. Also, for all $i\geq 1$, for all $\varepsilon>0$,
\[\mathbb P(L_i \geq \varepsilon (i\mathbb E[L]/2))
\leq \frac{\mathbb E[L^2]}{\varepsilon^2 (i\mathbb E[L]/2)^{2}},\]
which is summable.
Thus, Borel-Cantelli's lemma implies that, almost surely for all $i$ large enough, 
\[L_iT_i^{-1}\leq L_i (i\mathbb E[L]/2)^{-1}<\varepsilon,\]
as desired.
\end{proof}

\begin{lemma}\label{lem:prel_large}
For all $\delta>1$, almost surely
\[\sum_{i\geq 1} \frac{L^2_i}{T_{i-1}^\delta}<\infty.\]
\end{lemma}

\begin{proof}
First note that, by the strong law of large numbers, $T_i \sim i\mathbb E[L]$ almost surely as $i\uparrow\infty$. Thus, $\sum_{i\geq 1} L^2_i T_{i-1}^{-\delta}$ converges if and only if $\sum_{i\geq 1} L^2_i i^{-\delta}$ does. Now, we write
\[\sum_{i\geq 1} L^2_i i^{-\delta}
=  \mathbb E[L] \sum_{i\geq 1} i^{-\delta}
+ \sum_{i\geq 1} (L^2_i-\mathbb E[L^2]) i^{-\delta}.\]
The first sum above is convergent because $\delta >1$. For the second sum, note that
$(M_n := \sum_{i=1}^n (L^2_i-\mathbb E[L^2]) i^{-\delta})_{n\geq 1}$ is a martingale whose quadratic variation satisfies
\[\langle M\rangle_n =\sum_{i=1}^n \frac{\mathrm{Var}(L^2)}{i^{2\delta}}
\leq \mathbb E[L^4] \sum_{i\geq 1}\frac1{i^{2\delta}}<\infty,\]
because $\delta>1$.
\end{proof}

\begin{lemma}\label{lem:sum_power_Ti}
Almost surely as $i\uparrow\infty$,
\[\sum_{i=1}^n \frac1{T_i^{\delta-1}} =
\begin{cases}
\mathcal O(1) & \text{ if }\delta>2\\
\displaystyle \frac{n^{2-\delta}}{(2-\delta)\mathbb E[L]^{\delta-1}} + \mathcal O(1) & \text{ if }\delta\in (\nicefrac32, 2]\\[7pt]
\displaystyle \frac{n^{2-\delta}}{(2-\delta)\mathbb E[L]^{\delta-1}} + \mathcal O\big(n^{\frac32 - \delta}\sqrt{\log n}\big) & \text{ if }\delta\in (1, \nicefrac32].
\end{cases}
\]
\end{lemma}

\begin{lemma}\label{lem:sum_powerexp_Ti}
Almost surely as $i\uparrow\infty$,
\[\sum_{i=1}^n \frac{\mathrm e^{-\gamma\delta L_i T_i^{\delta-1}}}{T_i^{\delta-1}} =
\mathcal O(1).\]
\end{lemma}

\begin{proof}
First note that, for all $\delta>2$,
\[\sum_{i=1}^n \frac{\mathrm e^{-\gamma\delta L_i T_i^{\delta-1}}}{T_i^{\delta-1}} 
\leq \sum_{i=1}^n \frac1{T_i^{\delta-1}} <\infty,\]
by Lemma~\ref{lem:sum_power_Ti},
which concludes the proof.
We first note that, for all $\alpha\geq 1$, for all $x\geq 0$, $F_\alpha(x):=x^{\alpha}\mathrm e^{-x}\leq (\alpha/\mathrm e)^\alpha$.
Indeed, for all $x\geq 0$, $F'_{\alpha}(x) = (\alpha - x)x^{\alpha-1}\mathrm e^{-x}$ and thus $F$ reaches its maximum at $x= \alpha$ and this maximum is $F(\alpha) = (\alpha/\mathrm e)^\alpha$, as claimed.
Thus, for all $\alpha\geq 1$,
\[\sum_{i=1}^n \frac{\mathrm e^{-\gamma\delta L_i T_i^{\delta-1}}}{T_i^{\delta-1}}
\leq \Big(\frac\alpha{\mathrm e}\Big)^\alpha 
\sum_{i=1}^n \frac{1}{(\gamma\delta L_i)^\alpha T_i^{(\delta-1)(\alpha+1)}}
\]

\end{proof}

\begin{lemma}\label{lem:analysis}
For all $\delta>1$, for all $x\in[0,1]$,
\[1-\delta x\leq (1-x)^\delta\leq 1-\delta x+\delta(\delta-1) x^2\]
\end{lemma}

\subsection{Proof of Proposition \ref{Phi CLT delta > 1}}

To prove Proposition~\ref{Phi CLT delta > 1}, we first write
\[\Phi(n) 
= \sum_{i=1}^n \big(F_i{\bf 1}_{i\prec n}-\mathbb E_{\bf L}[F_i{\bf 1}_{i\prec n}]\big)+ \sum_{i=1}^n \mathbb E_{\bf L}\big[F_i{\bf 1}_{i\prec n}\big].\]
We treat the two sums in two separate lemmas:

\begin{lemma}\label{lem:large_lem_CLT}
Almost surely as $n\uparrow\infty$,
\[\sum_{i=1}^n \big(F_i{\bf 1}_{i\prec n}-\mathbb E_{\bf L}[F_i{\bf 1}_{i\prec n}]\big)
= \mathcal O(1).\]
\end{lemma}

\begin{proof}
Given {\bf L}, we let $(B_i)_{i\geq 1}$ be a sequence of independent Bernoulli-distributed random variables of respective parameters $W_i/\bar W_i$, independent of $(F_i)_{i\geq 1}$. 
Recall that, by Lemma~\ref{lem:magic}, for all $n\geq 1$, $({\bf 1}_{i\prec n})_{1\leq i\leq n} = (B_i)_{1\leq i\leq n}$.
Also note that, by definition, conditionally on ${\bf L}$,
\[\bigg(M_n :=\sum_{i=1}^n \big(F_i{\bf 1}_{i\prec n}-\mathbb E_{\bf L}[F_i{\bf 1}_{i\prec n}]\big)\bigg)_{n\geq 0}\]
is a martingale.
In the rest of the proof, we show that its quadratic variation satisfies
\begin{equation}\label{eq:large_cond_Petrov}
\langle M \rangle_n= \sum_{i\geq 1} \mathrm{Var}_{\bf L}(F_iB_i)<\infty.
\end{equation}
For all $i\geq 1$, by~\eqref{eq:total_variance},
\begin{equation}\label{eq:large_tot_var}
\mathrm{Var}_{\bf L}(F_i B_i) 
= \frac{W_i}{\bar W_i}\mathbb E_{\bf L}[F_i^2]-\bigg(\frac{W_i}{\bar W_i}\bigg)^2
\mathbb E_{\bf L}[F_i]^2 
%\leq \frac{W_i}{S_i}\mathbb E_{\bf L}[F_i^2].
\end{equation}
By definition of $F_i$ (see~\eqref{eq:def_F}) in the first equality, and integration by parts in the second,
\ba
W_i\mathbb E_{\bf L}[F_i^2]
&=\int_0^{L_i} x^2 \gamma\delta (T_{i-1}+x)^{\delta-1}\mathrm e^{\gamma(T_{i-1}+x)^\delta}\mathrm dx
= \big[x^2\mathrm e^{\gamma(T_{i-1}+x)^\delta}\big]_0^{L_i}
- \int_0^{L_i} 2x \mathrm e^{\gamma(T_{i-1}+x)^\delta}\mathrm dx\\
&= L_i^2 \bar W_i
-  \int_0^{L_i} 2x \mathrm e^{\gamma(T_i-(L_i-x))^\delta}\mathrm dx,
\ea
because $\bar W_i = \mathrm e^{\gamma T_i^\delta}$, by definition (see~\eqref{eq:def_W}).
We now use the fact that, for all $x\in[0,1]$, $(1-x)^\delta\geq 1-\delta x$, and get
\[W_i\mathbb E_{\bf L}[F_i^2]
\leq L_i^2 \bar W_i
-  \int_0^{L_i} 2x\,\mathrm{exp}\bigg(\gamma T_i^{\delta}\Big(1-\delta\cdot\frac{L_i-x}{T_i}\Big)\bigg)\mathrm dx
= L_i^2 \bar W_i
-  2\bar W_i \int_0^{L_i} x\, \mathrm{exp}\big(-\gamma\delta T_i^{\delta-1}(L_i-x)\big)\mathrm dx.
\]
Using integration by parts again, we get
\ba
&\int_0^{L_i} x\, \mathrm{exp}\big(-\gamma\delta T_i^{\delta-1}(L_i-x)\big)\mathrm dx\\
&\hspace{1cm}= \bigg[\frac{x\,\mathrm{exp}\big(-\gamma\delta T_i^{\delta-1}(L_i-x)\big)}{\gamma\delta T_i^{\delta-1}}\bigg]_0^{L_i}
- \frac1{\gamma\delta T_i^{\delta-1}}\int_0^{L_i}\mathrm{exp}\big(-\gamma\delta T_i^{\delta-1}(L_i-x)\big)\mathrm dx\\
&\hspace{1cm}= \frac{L_i}{\gamma\delta T_i^{\delta-1}} - \frac{1-\mathrm e^{-\gamma\delta L_i T_i^{\delta-1}}}{(\gamma\delta T_i^{\delta-1})^2}.
\ea
This implies
\begin{equation}\label{eq:large_tot_var1}
\frac{W_i\mathbb E_{\bf L}[F_i^2]}{\bar W_i}
\leq L_i^2
-  \frac{2L_i}{\gamma\delta T_i^{\delta -1}}
+ \frac1{(\gamma\delta T_i^{\delta-1})^2}.  
\end{equation}
Similarly,
\ban
W_i\mathbb E_{\bf L}[F_i] 
&= \int_0^{L_i} x \gamma\delta (T_{i-1}+x)^{\delta-1}\mathrm e^{\gamma(T_{i-1}+x)^\delta}\mathrm dx
= \big[x \mathrm e^{\gamma(T_{i-1}+x)^\delta}\big]_0^{L_i}
-  \int_0^{L_i} \mathrm e^{\gamma(T_{i-1}+x)^\delta}\mathrm dx\notag\\
&= L_i \bar W_i -   \int_0^{L_i} \mathrm e^{\gamma(T_{i-1}+x)^\delta}\mathrm dx
= L_i \bar W_i - \frac1\delta \int_{T_{i-1}^\delta}^{T_i^\delta} u^{\frac1\delta-1} \mathrm e^{\gamma u}\mathrm du,\label{eq:WEF}
\ean
where we have changed variables and set $u = (T_{i-1}+x)^\delta$.
We thus get
\begin{equation}\label{eq:WEF/S_lowerbound}
\frac{W_i\mathbb E_{\bf L}[F_i]}{\bar W_i}
\geq L_i - \frac1{\delta \bar W_i T_{i-1}^{\delta-1}}  \int_{T_{i-1}^\delta}^{T_i^\delta} \mathrm e^{\gamma u}\mathrm du
= L_i - \frac1{\gamma\delta T_{i-1}^{\delta-1}}.
\end{equation}
Using~\eqref{eq:large_tot_var} and~\eqref{eq:large_tot_var1}, we thus get
\ba
\mathrm{Var}_{\bf L}(F_i B_i)
&\leq L_i^2
-  \frac{2L_i}{\gamma\delta T_i^{\delta -1}}
+ \frac1{(\gamma\delta T_i^{\delta-1})^2} 
- \bigg( L_i - \frac1{\gamma\delta T_{i-1}^{\delta-1}}\bigg)^2\\
&= \frac{2L_i}{\gamma\delta} \bigg(\frac1{T_{i-1}^{\delta-1}}-\frac1{T_{i}^{\delta-1}}\bigg) +
\frac1{(\gamma\delta T_i^{\delta-1})^2} - \frac1{(\gamma\delta T_{i-1}^{\delta-1})^2}
\leq \frac{2L_i}{\gamma\delta T_{i-1}^{\delta-1}} \bigg(1-\Big(1-\frac{L_i}{T_i}\Big)^{\!\delta-1}\bigg)\\
&\leq \frac{2(\delta-1)L_i^2}{\gamma\delta T_{i-1}^{\delta-1}T_i}
\leq \frac{2(\delta-1)L_i^2}{\gamma\delta T_{i-1}^{\delta}},
\ea
where we have used the fact that, 
for all $x\in[0,1]$, $(1-x)^{\delta-1}\geq 1-(\delta-1) x$ (since $\delta>1$).
We thus get that
\[\sum_{i\geq 1} \mathrm{Var}_{\bf L}(F_i B_i) 
\leq \sum_{i\geq 1} \frac{2(\delta-1)L_i^2}{\gamma\delta T_{i-1}^{\delta}} 
<\infty,\]
by Lemma~\ref{lem:prel_large}.
Thus, \eqref{eq:large_cond_Petrov} holds and implies that $(M_n)_{n\geq 1}$ converges almost surely as $n\uparrow\infty$, which concludes the proof.
\end{proof}

\begin{lemma}\label{lem:large_lem_CLT2}
Almost surely as $n\uparrow\infty$,
\[\sum_{i=1}^n \mathbb E_{\bf L}[F_i{\bf 1}_{i\prec n}] 
= \begin{cases}
T_n + \mathcal O(1) & \text{ if }\delta >2\\[5pt]
\displaystyle T_n - \sum_{i=1}^n \frac{1}{\gamma\delta T_i^{\delta-1}} + \mathcal O(1)
& \text{ if }\delta \in (\nicefrac32, 2]\\
\displaystyle T_n - \sum_{i=1}^n \frac{1}{\gamma\delta T_i^{\delta-1}} + o\big(n^{\nicefrac32-\delta}(\log n)^2\big)
& \text{ if }\delta \in (1,\nicefrac32].
\end{cases}\]
\end{lemma}

\begin{proof}
First note that, for all $i\geq 1$, $\mathbb E_{\bf L}[F_i{\bf 1}_{i\prec n}\big]  = W_i\mathbb E_{\bf L}\big[F_i]/\bar W_i$, because, given $\bf L$,  $(F_i)_{1\leq i\leq n}$ is independent of $({\bf 1}_{i\prec n})$,
and because, by Lemma~\ref{lem:magic}, ${\bf 1}_{i\prec n}$ is Bernoulli-distributed with parameter $W_i/\bar W_i$.
By definition of $(F_i)_{i\geq 1}$ (see Definition~\eqref{eq:def_F}) in the first equality, and integration by parts in the second, we get
\ba
\frac{W_i\mathbb E_{\bf L}\big[F_i]}{\bar W_i}
&= \int_0^{L_i} \gamma\delta x(T_{i-1}+x)^{\delta-1} 
\mathrm e^{-\gamma [T_i^\delta -(T_{i-1}+x)^\delta]}\mathrm dx\\
&= L_i -  \int_0^{L_i} \mathrm e^{-\gamma [T_i^\delta -(T_{i-1}+x)^\delta]}\mathrm dx
= L_i -  \int_0^{L_i} \mathrm e^{-\gamma T_i^{\delta} [1 -(1-\nicefrac u{T_i})^\delta]}\mathrm du,
\ea
where we have changed variables and set $u = L_i-x$.
This means that
\[\sum_{i=1}^n\mathbb E_{\bf L}[F_i{\bf 1}_{i\prec n}\big]  
= T_n - \sum_{i=1}^n \int_0^{L_i} \mathrm e^{-\gamma T_i^{\delta} [1 -(1-\nicefrac u{T_i})^\delta]}\mathrm du.
\]
Thus, it only remains to show that
\begin{equation}\label{eq:large_aim_diff}
\sum_{i=1}^n \int_0^{L_i} \mathrm e^{-\gamma T_i^{\delta} [1 -(1-\nicefrac u{T_i})^\delta]}\mathrm du
= \begin{cases}
\mathcal O(1)& \text{ if }\delta >2\\[7pt]
\displaystyle\sum_{i=1}^n \frac1{\gamma\delta T_i^{\delta-1}} + \mathcal O(1)
& \text{ if }\delta \in (\nicefrac32, 2]\\
\displaystyle\sum_{i=1}^n \frac1{\gamma\delta T_i^{\delta-1}} + o\big(n^{\nicefrac32-\delta}(\log n)^2\big)
& \text{ if }\delta \in (1, \nicefrac32],
\end{cases}
\end{equation}
almost surely as $n\uparrow\infty$.
%To prove this, we first show that, almost surely as $n\uparrow\infty$,
%\begin{equation}\label{eq:large_aim1}
%\sum_{i=1}^n \int_0^{L_i} \mathrm e^{-\gamma T_i^{\delta} [1 -(1-\nicefrac u{T_i})^\delta]}\mathrm du
%=\sum_{i=1}^n \frac{1}{\gamma\delta T_i^{\delta-1}} + \mathcal O(1).
%\end{equation}
By Lemma~\ref{lem:analysis}, because $u\leq L_i\leq T_i$ almost surely,
\ban
\int_0^{L_i} \mathrm e^{-\gamma\delta u T_i^{\delta-1}}\mathrm du
\leq \int_0^{L_i} \mathrm e^{-\gamma T_i^{\delta} [1 -(1-\nicefrac u{T_i})^\delta]}\mathrm du
&\leq \int_0^{L_i} \mathrm e^{-\gamma\delta u T_i^{\delta-1}(1-(\delta-1)\nicefrac u{T_i})}\mathrm du
\label{eq:large_LB_toimprove}\\
&\leq \int_0^{L_i} \mathrm e^{-\gamma\delta u T_i^{\delta-1}(1-(\delta-1)\nicefrac{L_i}{T_i})}\mathrm du.\notag
\ean
In the last inequality, we have used the fact that $u\leq L_i$ and $\delta>1$.
This implies
\[\frac{1-\mathrm e^{-\gamma\delta L_i T_i^{\delta-1}}}{\gamma\delta T_i^{\delta-1}}
\leq \int_0^{L_i} \mathrm e^{-\gamma T_i^{\delta} [1 -(1-\nicefrac u{T_i})^\delta]}\mathrm du
\leq \frac1{\gamma\delta T_i^{\delta-1}(1-(\delta-1)L_i/T_i)}.\]
Note that, by Lemma~\ref{lem:large_BC}, almost surely, $\frac{L_i}{T_i}<\frac1{2(\delta-1)}$ for all $i$ large enough.
Using the fact that $\frac1{1-u}\leq 1+2u$ for all $u\in[0,\nicefrac12]$, we thus get that, almost surely for all $i$ large enough,
\[\frac{1-\mathrm e^{-\gamma\delta L_i T_i^{\delta-1}}}{\gamma\delta T_i^{\delta-1}}
\leq \int_0^{L_i} \mathrm e^{-\gamma T_i^{\delta} [1 -(1-\nicefrac u{T_i})^\delta]}\mathrm du
\leq \frac{1}{\gamma\delta T_i^{\delta-1}}+\frac{2(\delta-1)L_i}{\gamma\delta T_i^{\delta}}.\]
This implies that
\begin{equation}\label{eq:double_bound}
-\sum_{i=1}^n \frac{2(\delta-1)L_i}{\gamma\delta T_i^{\delta}}
\leq \sum_{i=1}^n \frac{1}{\gamma\delta T_i^{\delta-1}} - \sum_{i=1}^n \int_0^{L_i} \mathrm e^{-\gamma T_i^{\delta} [1 -(1-\nicefrac u{T_i})^\delta]}\mathrm du
\leq \sum_{i=1}^n \frac{\mathrm e^{-\gamma\delta L_i T_i^{\delta-1}}}{\gamma\delta T_i^{\delta-1}}.
\end{equation}
By Lemma~\ref{lem:prel_large}, the left-hand side is $\mathcal O(1)$ as $n\uparrow\infty$.
Furthermore, by Lemma~\ref{lem:prel_large}, $ \sum_{i=1}^n {1}/({\gamma\delta T_i^{\delta-1}})<\infty$ if $\delta>2$.
Thus, to prove~\eqref{eq:large_aim_diff},
it is enough to prove that
\begin{equation}\label{eq:large_aim2}
\sum_{i=1}^n \frac{\mathrm e^{-\gamma\delta L_i T_i^{\delta-1}}}{\gamma\delta T_i^{\delta-1}} 
= \begin{cases}
\mathcal O(1) & \text{ if }\delta>\nicefrac32\\
o\big(n^{\nicefrac32-\delta}(\log n)^2\big) & \text{ if }\delta\in (1,\nicefrac32],
\end{cases}
\end{equation}
almost surely as $n\uparrow\infty$.
If $\delta>2$, this is implied by Lemma~\ref{lem:prel_large}.
For $\delta\in (1, 2]$, 
we first use the strong law of large numbers to write that $T_i\geq i\mathbb EL/2$ almost surely for all $i$ large enough, and thus
\begin{equation}\label{eq:large_towards_aim2}
\sum_{i=1}^n \frac{\mathrm e^{-\gamma\delta L_i T_i^{\delta-1}}}{\gamma\delta T_i^{\delta-1}}
= \mathcal O\bigg(\sum_{i=1}^n \frac{\mathrm e^{-\gamma\delta L_i (i\mathbb EL/2)^{\delta-1}}}
{\gamma\delta (i\mathbb EL/2)^{\delta-1}}\bigg)
= \mathcal O\bigg(\sum_{i=1}^n \frac{\mathrm e^{-\gamma\delta L_i (i\mathbb EL/2)^{\delta-1}}}
{i^{\delta-1}}\bigg),
\end{equation}
almost surely as $n\uparrow\infty$.
First note that $(M_n := \sum_{i=1}^n \mathrm e^{-\gamma\delta L_i (i\mathbb EL/2)^{\delta-1}}{i^{1-\delta}})_{n\geq 0}$ is a martingale whose quadratic variation satisfies
\begin{equation}\label{eq:angle}
\langle M\rangle_n
= \sum_{i=1}^n \frac{\mathbb E[\mathrm e^{-2\gamma\delta L_i (i\mathbb EL/2)^{\delta-1}}]}
{i^{2(\delta-1)}}.
\end{equation}
If $\delta>\nicefrac32$, then $2(\delta -1)>1$ implying that
\[\langle M\rangle_n\leq \sum_{i\geq 1} 
\frac1{i^{2(\delta-1)}}<\infty,
\]
and thus that $(M_n)_{n\geq 0}$ converges almost surely as $n\uparrow\infty$. 
This, together with~\eqref{eq:large_towards_aim2}, concludes the proof of \eqref{eq:large_aim2} in the case when $\delta\in (\nicefrac32, 2]$.
For $\delta\in (1, \nicefrac32]$, we distinguish two cases:
First assume that $\lim_{n\uparrow\infty} \langle M\rangle_n<\infty$. In his case, $M_n$ converges almost surely as $n\uparrow\infty$, which, together with~\eqref{eq:large_towards_aim2}, implies that
\[\sum_{i=1}^n \frac{\mathrm e^{-\gamma\delta L_i T_i^{\delta-1}}}{\gamma\delta T_i^{\delta-1}}
= \mathcal O(1) = o\big(n^{\nicefrac32-\delta}(\log n)^2\big),\]
as claimed in~\eqref{eq:large_aim2}.
If $\lim_{n\uparrow\infty} \langle M\rangle_n = \infty$, then by, e.g.~\cite[Theorem~1.3.15]{Duflo},
almost surely as $n\uparrow\infty$,
\[M_n = o\big(\sqrt{\langle M\rangle_n}(\log n)^2\big).\]
By~\eqref{eq:angle},
\[\langle M\rangle_n\leq \sum_{i=1}^n \frac1{i^{2(\delta-1)}} = \mathcal O(n^{3-2\delta}),\]
which gives that, almost surely as $n\uparrow\infty$, $M_n = o\big(n^{\nicefrac32-\delta}(\log n)^2\big)$. Together with~\eqref{eq:large_towards_aim2}, this implies that
\[\sum_{i=1}^n \frac{\mathrm e^{-\gamma\delta L_i T_i^{\delta-1}}}{\gamma\delta T_i^{\delta-1}}
= o\big(n^{\nicefrac32-\delta}(\log n)^2\big),\]
which concludes the proof of~\eqref{eq:large_aim2}.

Note that~\eqref{eq:large_aim2}, together with~\eqref{eq:double_bound}, implies~\eqref{eq:large_aim_diff}.
To conclude the proof of~\eqref{eq:large_aim_diff}, 
it only remains to prove that, for all $\delta>2$, almost surely as $n\uparrow\infty$,
\begin{equation}\label{eq:large_aim3}
\sum_{i=1}^n \frac{1}{\gamma\delta T_i^{\delta-1}} 
= \mathcal O(1)
\end{equation}
This is true because $T_i \sim i\mathbb EL$ almost surely as $i\uparrow\infty$.
\end{proof}

By Lemmas~\ref{lem:large_lem_CLT} and~\ref{lem:large_lem_CLT}, almost surely as $n\uparrow\infty$,
\ba
\Phi(n) 
&= \sum_{i=1}^n \big(F_i{\bf 1}_{i\prec n}-\mathbb E_{\bf L}[F_i{\bf 1}_{i\prec n}]\big)+ \sum_{i=1}^n \mathbb E_{\bf L}\big[F_i{\bf 1}_{i\prec n}\big]\\
&= \begin{cases}
T_n + \mathcal O(1) & \text{ if }\delta >2\\[5pt]
\displaystyle T_n - \sum_{i=1}^n \frac1{\gamma\delta T_i^{\delta-1}} + \mathcal O(1)
& \text{ if }\delta \in (\nicefrac32, 2]\\[5pt]
\displaystyle T_n - \sum_{i=1}^n \frac1{\gamma\delta T_i^{\delta-1}} + o\big(n^{\nicefrac32-\delta}(\log n)^2\big)
& \text{ if }\delta \in (1, \nicefrac32],
\end{cases}
\ea
as claimed in Proposition \ref{Phi CLT delta > 1}.

\bibliographystyle{plain}   
\bibliography{Citations}

\begin{thebibliography}{10}

\bibitem{Bertoin}
Jean Bertoin, Klaas Van~Harn, and Frederik~Willem Steutel.
\newblock Renewal theory and level passage by subordinators.
\newblock {\em Statistics \& probability letters}, 45(1):65--69, 1999.

\bibitem{BM}
Erion-Stelios Boci and Cecile Mailler.
\newblock Large deviation principle for a stochastic process with random
  reinforced relocations.
\newblock {\em Journal of Statistical Mechanics: Theory and Experiment},
  2023(8):083206, 2023.

\bibitem{BV06}
Konstantin~A. Borovkov and Vladimir~A. Vatutin.
\newblock On the asymptotic behaviour of random recursive trees in random
  environments.
\newblock {\em Advances in applied probability}, 38(4):1047--1070, 2006.

\bibitem{BV05}
Konstantin~A. Borovkov and Vladimir~A. Vatutin.
\newblock Trees with product-form random weights.
\newblock In {\em Fourth Colloquium on Mathematics and Computer Science
  Algorithms, Trees, Combinatorics and Probabilities}, pages 423--426. Discrete
  Mathematics and Theoretical Computer Science, 2006.

\bibitem{BEM17}
Denis Boyer, Martin~R. Evans, and Satya~N. Majumdar.
\newblock Long time scaling behaviour for diffusion with resetting and memory.
\newblock {\em Journal of Statistical Mechanics: Theory and Experiment},
  2017(2):023208, Sep 2017.

\bibitem{BFGM19}
Denis Boyer, Andrea Falc{\'o}n-Cort{\'e}s, Luca Giuggioli, and Satya~N
  Majumdar.
\newblock Anderson-like localization transition of random walks with resetting.
\newblock {\em Journal of Statistical Mechanics: Theory and Experiment},
  2019(5):053204, 2019.

\bibitem{BM24}
Denis Boyer and Satya~N Majumdar.
\newblock Active particle in one dimension subjected to resetting with memory.
\newblock {\em Physical Review E}, 109(5):054105, 2024.

\bibitem{BP16}
Denis Boyer and Inti Pineda.
\newblock Slow {L}\'evy flights.
\newblock {\em Physical Review E}, 93:022103, Feb 2016.

\bibitem{BSS14}
Denis Boyer and Citlali Solis-Salas.
\newblock Random walks with preferential relocations to places visited in the
  past and their application to biology.
\newblock {\em Physical Review Letters}, 112(24), 2014.

\bibitem{Duflo}
Marie Duflo.
\newblock {\em Random Iterative Methods}.
\newblock Springer Verlag, 1997.

\bibitem{EMS}
Martin~R Evans, Satya~N Majumdar, and Gr{\'e}gory Schehr.
\newblock Stochastic resetting and applications.
\newblock {\em Journal of Physics A: Mathematical and Theoretical},
  53(19):193001, 2020.

\bibitem{MUB19}
C{\'e}cile Mailler and Ger{\'o}nimo~Uribe Bravo.
\newblock Random walks with preferential relocations and fading memory: a study
  through random recursive trees.
\newblock {\em Journal of Statistical Mechanics: Theory and Experiment},
  2019(9):093206, 2019.

\bibitem{MM17}
C{\'e}cile Mailler and Jean-Fran{\c{c}}ois Marckert.
\newblock Measure-valued {P}{\'o}lya processes.
\newblock {\em Electronic Journal of Probability}, 22:26, 2017.

\bibitem{PS}
Michel Pain and Delphin S{\'e}nizergues.
\newblock Correction terms for the height of weighted recursive trees.
\newblock {\em The Annals of Applied Probability}, 32(4):3027--3059, 2022.

\bibitem{petrov}
Valentin~V. Petrov.
\newblock {\em Sums of independent random variables}.
\newblock Springer Verlag, 1975.

\bibitem{D19}
Delphin S{\'e}nizergues.
\newblock Geometry of weighted recursive and affine preferential attachment
  trees.
\newblock {\em Electronic Journal of Probability}, 26:1--56, 2021.

\end{thebibliography}

\end{document}